\documentclass[11pt]{article}

\usepackage{amsmath,amssymb,vmargin,color,hyperref,epsfig}
\usepackage{tikz}
\usetikzlibrary{decorations.markings}

\newcommand{\ee}{\textrm{e}}
\newcommand{\ii}{\textrm{i}}
\newcommand{\dd}{\textrm{d}}

\newtheorem{theorem}{Theorem}
\newtheorem{proposition}{Proposition}
\newtheorem{corollary}{Corollary}
\newtheorem{lemma}{Lemma}
\newtheorem{proof}{Proof}
\newtheorem{remark}{Remark}
\newtheorem{example}{Example}
\newtheorem{assumption}{Assumption}

\numberwithin{equation}{section}
\numberwithin{theorem}{section}
\numberwithin{corollary}{section}
\numberwithin{proposition}{section}
\numberwithin{example}{section}
\numberwithin{remark}{section}
\numberwithin{proof}{section}

\title{Asymptotics of matrix orthogonal polynomials on the real line}

\author{
Alfredo Dea\~{n}o\footnote{Universidad Carlos III de Madrid, https://ror.org/03ths8210, Departamento de Matem\'aticas, Avenida de la Universidad, 30 (edificio Sabatini), 28911 Legan\'es (Madrid), Espa\~{n}a. Email: alfredo.deanho@uc3m.es} , 
Pablo Rom\'an\footnote{CIEM-CONICET and FaMAF-Universidad Nacional de C\'ordoba, Ciudad Universitaria, C\'ordoba, Argentina;
Guangdong Technion Israel Institute of Technology, 241 Daxue Road, Jinping District, Shantou, Guandong Province, China. Email: pablo.roman@unc.edu.ar}}

\date{\today}
\begin{document}
\maketitle

\begin{abstract}
In this paper, we are interested in matrix valued orthogonal polynomials on the real line with respect to exponential weights. We obtain strong asymptotics as the degree tends to infinity in different regions of the complex plane, as well as asymptotic behavior of recurrence coefficients and norms. The main tools are the Riemann--Hilbert formulation and the Deift--Zhou method of steepest descent, adapted to the matrix case. A central role is played by the matrix Szeg\H{o} function, an object that has independent interest.
\end{abstract}
\textbf{MSC classification numbers:} 33C45, 34M50, 30E15, 41A60

\section{Introduction}

In this paper, we study matrix orthogonal polynomials (MVOPs in the sequel) with respect to exponential weight functions on the real line. More precisely, we consider an $r\times r$ weight of the following form:
\begin{equation}\label{eq:Wx}
W_N(x)
=
\ee^{-Nv(x)} M(x)
\qquad
v(x)
=
x^{2m}+\sum_{j=0}^{2m-1}v_jx^j, \qquad m\geq 1.
\end{equation}
where $M(x)$ is an $r\times r$ Hermitian positive definite matrix 
on the real line, and $N>0$ is a parameter.

We can write the matrix part of the weight in the following form:
\begin{equation}\label{eq:MQQ}
M(x)=Q(x)Q(x)^\ast,
\end{equation}
where $Q(x)$ is a matrix valued function, that is analytic for $x\in\mathbb{R}$ and admits an analytic extension $Q(z)$ to a neighborhood of $\mathbb{R}$. 

A particularly important example is given by $Q(x)=\ee^{Ax}$, where the matrix $A$ is of size $r\times r$, nilpotent and with the following structure:
\begin{equation}\label{eq:A}
A
=
\begin{pmatrix}
    0 & 0 & \ldots & \ldots & 0\\
    \alpha_1 & 0 &\ddots & &\vdots\\
    0 & \alpha_2 & 0 &\ddots & \vdots \\
    \vdots & \ddots & \ddots & \ddots & 0 \\
    0 &\ldots & 0 & \alpha_{r-1} & 0
\end{pmatrix}
\end{equation}
with $\alpha_j\in\mathbb{R}$ for $j=1,\ldots,r-1$. A simpler case is obtained by letting $\alpha_j=\alpha$ for all $j$. 

Given this weight function, we construct monic MVOPs $P_{n,N}(x)=x^n I_r+\ldots$, which depend on the two parameters $n,N$ and the variable $x$ and satisfy the orthogonality conditions
\begin{equation}\label{eq:MVOPs}
\int_{\mathbb{R}} P_{n,N}(x)W_N(x)P_{m,N}(x)^\ast \dd x
=
\delta_{n,m}\mathcal{H}_{n,N},
\end{equation}
where $\mathcal{H}_{n,N}$ is a Hermitian positive definite matrix and $\delta_{n,m}$ is Kronecker's delta. 

In the case of Hermite scalar factor, that is $v(x)=x^2$, the corresponding MVOPs appear in the context of differential equations of order $2$ with matrix coefficients having MVOPs as eigenfunctions in the papers \cite{DG2004, DG2005,DG2007} by Dur\'an and Gr\"{u}nbaum, see also \cite[Chapter 12]{AB2020} for a recent overview. In the $2\times 2$ Hermite case, the MVOPs can be written directly in terms of scalar Hermite polynomials, and many important properties (such as Rodrigues formula, recurrence coefficients etc) can be obtained explicitly. The Hermite and Freud (when $v(x)=x^4$) cases are also studied by Gr\"{u}nbaum, D. de la Iglesia and Mart\'inez--Finkelshtein in \cite{GdIM2012}, as well as by Cassatella--Contra and Ma\~{n}as in \cite{CCM2012}. In these references, the authors present a Riemann--Hilbert (RH) problem for MVOPs on the real line (that we will use later), with a particular interest in the algebraic and differential identities that these MVOPs satisfy. Such identities are also considered in \cite{DER2021}, in the light of the recent paper by Casper and Yakimov \cite{CY2022}, they correspond to non-commutative versions of known identities in integrable systems, such as Toda lattice or Painlev\'e difference equations.

In this paper, we present asymptotic results for MVOPs $P_{n,N}(x)$ with respect to a matrix weight of the form \eqref{eq:Wx}, as $N$ and the degree $n$ tend to infinity. 

\begin{assumption}\label{assumption1}
\normalfont
In our asymptotic analysis we make the following assumptions on the matrix weight function $W_N(x)$:
    \begin{itemize}
        \item The potential $v(x)$ is such that the corresponding equilibrium measure is supported on a single interval of the real line, that we denote $[a,b]$. This is true for instance if $v(x)$ is convex, and a standard example is given by the monomial $v(x)=x^{2m}$, with $m\geq 1$. The case $m=1$ gives the classical Hermite weight. This interval $[a,b]$ can be determined using the Mhaskar--Rakhmanov--Saff (MRS) numbers, which in this setting solve the following equations:
        \begin{equation}\label{eq:MRS}
        \frac{1}{2\pi} \int_{a}^{b} v'(s)\sqrt{\frac{s-a}{b-s}}\dd s=1, \qquad
        \frac{1}{2\pi} \int_{a}^{b} v'(s)\sqrt{\frac{b-s}{s-a}}\dd s=-1,
        \end{equation}
we refer the reader to \cite{DKMVZ1999,LL2001,ST1997}. 

From the MRS numbers \eqref{eq:MRS}, we define the following parameters:
\begin{equation}\label{eq:ctdt}
c=\frac{b-a}{2},\qquad d=\frac{b+a}{2},
\end{equation}

\item The matrix part of the weight $M(x)=Q(x)Q(x)^\ast$ in \eqref{eq:MQQ} 
remains independent of $N$. 
\end{itemize}
\end{assumption}

As a consequence of the second assumption, several steps of the steepest descent analysis can be directly adapted from the literature on the scalar case, see for example \cite{DKMVZ1999,KT2009}. One important exception is the construction of the global parametrix, which requires a matrix Szeg\H{o} function. Such a Szeg\H{o} function is an essential tool in the asymptotic analysis of MVOPs on the interval $[-1,1]$, presented recently in \cite{DKR2023}. For exponential weights on the real line as \eqref{eq:Wx}, the main idea is similar, and in the case $Q(x)=\ee^{Ax}$, the construction can be carried out algorithmically, following ideas in \cite{ephremidze_2014}. We also note that in this context, unlike in the situation studied in \cite{DKR2023}, there is no need for spectral decomposition of the weight function.

\section{Riemann--Hilbert problem}\label{sec:RHid}
These MVOPs arise as solution of a Riemann--Hilbert (RH) problem of size $2r\times 2r$, we refer the reader to \cite{CCM2012,GdIM2012}: we seek $Y(z)=Y(z,n,N,\{\alpha_j\}):\mathbb{C}\mapsto\mathbb{C}^{2r\times 2r}$ such that
\begin{enumerate}
\item $Y(z)$ is analytic in $\mathbb{C}\setminus \mathbb{R}$.
\item For $x\in\mathbb{R}$, the matrix $Y(z)$ has continuous boundary values $Y_{\pm}(x)=\lim_{\varepsilon\to 0}Y(x\pm\ii\varepsilon)$, that satisfy the following jump condition:
\begin{equation}
Y_+(x)
=
Y_-(x)
\begin{pmatrix}
I_r & W_N(x)\\
0_r & I_r
\end{pmatrix},
\qquad W_N(x)=\ee^{-Nv(x)}M(x).
\end{equation}
Here and in the sequel, $0_r$ and $I_r$ indicate the zero and identity matrix of size $r\times r$, respectively.
\item As $z\to\infty$, we have the asymptotic behavior
\begin{equation}\label{eq:asympYN}
Y(z)
=
\left(I_{2r}+\frac{Y^{(1)}}{z}+\frac{Y^{(2)}}{z^2}+\mathcal{O}(z^{-2})\right) 
\begin{pmatrix} z^n I_r & 0_r\\ 0_r &  z^{-n} I_r\end{pmatrix}
\end{equation}
\end{enumerate}

\begin{remark}\label{rem:sigma3}
\normalfont
    In the sequel, we will use the standard notation with the Pauli $\sigma_3$ matrix, but understood by blocks in this matrix setting. That is, given $f(z)\neq 0$, we write
    \begin{equation}
        f(z)^{\sigma_3}
        :=
        \begin{pmatrix} f(z) I_r & 0_r\\ 0_r &  \frac{1}{f(z)} I_r\end{pmatrix}.
    \end{equation}
\end{remark}

The solution of this Riemann--Hilbert problem is given by the following matrix (written in $r\times r$ blocks):
\begin{equation}\label{eq:solRHPY}
Y(z)
=
\begin{pmatrix}
P_{n,N}(z) & \mathcal{C}(P_{n,N}W_N)(z)\\
-2\pi \ii \mathcal{H}_{n-1,N}^{-1}P_{n-1,N}(z) & -2\pi \ii \mathcal{H}_{n-1,N}^{-1}\mathcal{C}(P_{n-1,N}W_N)(z)
\end{pmatrix},
\end{equation}
where 
\begin{equation}
\mathcal{C}(f)(z)=\frac{1}{2\pi\ii}\int_{\mathbb{R}}\frac{f(s)}{s-z}\dd s
\end{equation}
is a standard Cauchy transform (taken entrywise), and the norm $\mathcal{H}_{n,N}$ is given by \eqref{eq:MVOPs}.

As a direct consequence of orthogonality \eqref{eq:MVOPs}, the  MVOPs $P_{n,N}(z)$ satisfy a three term recurrence relation 
\begin{equation}\label{eq:TTRRN}
zP_{n,N}(z)
=
P_{n+1,N}(z)
+B_{n,N} P_{n,N}(z)
+C_{n,N} P_{n-1,N}(z),
\end{equation}
where the coefficients $B_{n,N}$ and $C_{n,N}$ are $r\times r$ matrices. For simplicity, in the sequel we omit writing the dependence on $N$ explicitly, and we understand $P_n(z)=P_{n,N}(z)$ etc.

It is well known (see for example \cite[\S 4.1]{GdIM2012}) that these recurrence coefficients can be written in terms of the entries of the matrices $Y^{(1)}$ and $Y^{(2)}$ that appear in the large $z$ asymptotic expansion \eqref{eq:asympYN} of the Riemann--Hilbert problem. We write these formulas below for completeness, but also noting that the second expression for $B_{n}$, which is analogous to the one in \cite[Theorem 3.1]{DKMVZ1999} for the scalar case, appears to be new in the matrix setting.

\begin{proposition}
\normalfont
The recurrence coefficients $B_{n}$ and $C_{n}$ can be written as follows:
\begin{equation}\label{eq:BnCn}
B_{n}=Y^{(1)}_{n,11}-Y^{(1)}_{n+1,11}, \qquad
C_{n}=Y_{n,12}^{(1)}Y_{n,21}^{(1)}.
\end{equation}
Furthermore, the coefficient $B_{n}$ admits the following reformulation:
\begin{equation}\label{eq:BnY1Y2}
    B_{n}
=
Y_{n,11}^{(1)}
-
\left(Y_{n,12}^{(2)}\right)^\ast
\left(Y_{n,12}^{(1)}\right)^{-1}.
\end{equation}
\end{proposition}
 
\begin{proof}
\normalfont
The proof of the formula for $B_{n}$ in \eqref{eq:BnCn} follows directly from equating terms multiplying $z^n$ in the recurrence relation \eqref{eq:TTRRN}, and the fact that if $P_n(z)=z^n I_r+X_{n,n-1}z^{n-1}+\ldots$, then the subleading coefficient is $X_{n,n-1}=Y_{11}^{(1)}$, from \eqref{eq:asympYN}. 

Regarding the formula for $C_{n}$, we have $C_{n}=\mathcal{H}_n\mathcal{H}_{n-1}^{-1}$, in terms of norms in \eqref{eq:MVOPs} and directly from orthogonality. On the other hand, we have
\begin{equation}\label{eq:Y12infty}
\begin{aligned}
    Y_{n,12}(z)
    &=
    \frac{1}{2\pi\ii}\int_{\mathbb{R}}\frac{P_n(s)W(s)}{s-z}\dd s\\
    &=
    -\frac{1}{2\pi\ii z^{n+1}}\int_{\mathbb{R}}
    P_n(s)W(s)s^n\dd s
    -\frac{1}{2\pi\ii z^{n+2}}\int_{\mathbb{R}}
    P_n(s)W(s)s^{n+1}\dd s+\mathcal{O}(z^{-n-3}),
\end{aligned}
\end{equation}
as $z\to\infty$, $z\notin\mathbb{R}$, using orthogonality. Then, 
\begin{equation}
\begin{aligned}
    \int_{\mathbb{R}}
    P_n(s)W(s)s^n\dd s
    &=
    \int_{\mathbb{R}}
    P_n(s)W(s)(s^n I_r)^\ast\dd s\\
    &=
    \int_{\mathbb{R}}
    P_n(s)W(s)\left(P_n(s)-X_{n,n-1}\left(P_{n-1}(s)-\ldots\right)\right)^\ast\dd s\\
    &=
    \mathcal{H}_n,
\end{aligned}
\end{equation}
which means that $Y^{(1)}_{12}=-(2\pi\ii)^{-1}\mathcal{H}_{n}$. Also, since $Y^{(1)}_{21}(z)=-2\pi\ii \mathcal{H}_{n-1} \left(z^{n-1}+\ldots\right)$, we obtain $Y_{21}^{(1)}=-2\pi\ii \mathcal{H}_{n-1}^{-1}$, and that leads to \eqref{eq:BnCn} for $C_{n}$.

In order to prove \eqref{eq:BnY1Y2}, we examine the next term in the asymptotic expansion \eqref{eq:Y12infty}, and we observe that 
\begin{equation}
\begin{aligned}
    Y_{12}^{(2)}
    &=-\frac{1}{2\pi\ii}
    \int_{\mathbb{R}}
    P_n(s)W(s)s^{n+1}\dd s\\
    &=
    -\frac{1}{2\pi\ii}
    \int_{\mathbb{R}}
    P_n(s)W(s)(s^{n+1})^\ast\dd s\\
    &=
    -\frac{1}{2\pi\ii}
    \int_{\mathbb{R}}
    P_n(s)W(s)\left(P_{n+1}(s)-X_{n+1,n}\left(P_{n}(s)-\ldots\right)\right)^\ast\dd s=
    \frac{1}{2\pi\ii}\mathcal{H}_n X_{n+1,n}^\ast.
\end{aligned}
\end{equation}

Since $\mathcal{H}_n=-2\pi\ii Y_{n,12}^{(1)}$, it follows that $Y_{n,12}^{(2)}
=
-Y_{n,12}^{(1)} X_{n+1,n}^\ast 
$, and then 
\begin{equation}
X_{n+1,n}
=
-\left(\left(Y_{n,12}^{(1)}\right)^{-1}Y_{n,12}^{(2)}\right)^\ast 
=
-\left(Y_{n,12}^{(2)}\right)^\ast
\left(Y_{n,12}^{(1)}\right)^{-\ast}
=
\left(Y_{n,12}^{(2)}\right)^\ast
\left(Y_{n,12}^{(1)}\right)^{-1}.
\end{equation}
Here and in the sequel, we write $M^{-\ast}=(M^{-1})^{\ast}$ for any invertible matrix. The last equality follows from
\begin{equation}
Y_{n,12}^{(1)}
=
-\frac{1}{2\pi\ii} \mathcal{H}_n
\Rightarrow
\left(Y_{n,12}^{(1)}\right)^{-\ast}
=
2\pi\ii \left(\mathcal{H}_n\right)^{-\ast}
=
2\pi\ii \mathcal{H}_n^{-1}
=
-\left(Y_{n,12}^{(1)}\right)^{-1},
\end{equation}
because $\mathcal{H}_n$ is Hermitian. Writing everything together for $X_{n,n-1}-X_{n+1,n}$ and using \eqref{eq:BnCn}, 
we obtain \eqref{eq:BnY1Y2}.

\end{proof}

\begin{remark}\label{rem:orthonorm}
\normalfont
The recurrence relation can be rewritten in a more symmetric form for orthonormal polynomials. We define
\begin{equation}
    \Pi_{n,N}(z)=\kappa_{n,N} P_{n,N}(z),
\end{equation}
where $\kappa_{n,N}$ is a non-singular matrix chosen so that
\begin{equation}
\int_{\mathbb{R}} \Pi_{n,N}(x)W_N(x)\Pi_{n,N}(x)^\ast \,\dd x=I_r.
\end{equation}
Equivalently,
\begin{equation}
\kappa_{n,N}\mathcal{H}_{n,N}\kappa_{n,N}^\ast=I_r,
\end{equation}
hence
\begin{equation}
\mathcal{H}_{n,N}=\kappa_{n,N}^{-1}\kappa_{n,N}^{-\ast}=(\kappa_{n,N}^\ast \kappa_{n,N})^{-1}.
\end{equation}
This factorization is not unique: if $\widetilde{\kappa}_{n,N}=U\kappa_{n,N}$ with $U$ unitary, then
\begin{equation}
\widetilde{\kappa}_{n,N}^{-1}\widetilde{\kappa}_{n,N}^{-\ast}=\kappa_{n,N}^{-1}\kappa_{n,N}^{-\ast}.
\end{equation}

The recurrence relation \eqref{eq:TTRRN} becomes
\begin{equation}\label{eq:TTRR_orthonorm}
z\Pi_{n,N}(z)
=
A_{n+1,N}\Pi_{n+1,N}(z)
+\kappa_{n,N} B_{n,N} \kappa_{n,N}^{-1} \Pi_{n,N}(z)
+A_{n,N}^\ast \Pi_{n-1,N}(z),
\end{equation}
where 
\begin{equation}
A_{n,N}=\kappa_{n-1,N}\kappa_{n,N}^{-1}.
\end{equation}
Moreover, we observe that the coefficient  $\widehat{B}_{n,N}:=\kappa_{n,N} B_{n,N} \kappa_{n,N}^{-1}$ is Hermitian (unlike $B_{n,N}$ in general). 
Indeed, from orthogonality one has $B_{n,N}\mathcal{H}_{n,N}=\mathcal{H}_{n,N}B_{n,N}^\ast$, and using $\mathcal{H}_{n,N}^{-1}=\kappa_{n,N}^\ast\kappa_{n,N}$ we obtain
\begin{equation}
\widehat{B}_{n,N}^\ast
=
\kappa_{n,N}^{-\ast}B_{n,N}^\ast\kappa_{n,N}^\ast
=
\kappa_{n,N}B_{n,N}\kappa_{n,N}^{-1}
=
\widehat{B}_{n,N}.
\end{equation}
\end{remark}

\section{Matrix Szeg\H{o} function}
Our results make use of the matrix Szeg\H{o} function $D(z)$, which appears in a suitable factorization of the weight:
\begin{proposition}\label{prop:SzegoD}
\normalfont
Given a weight matrix $M(x)=Q(x)Q(x)^\ast$ that is Hermitian positive definite for $x\in(-1,1)$, and such that 
\begin{equation}\label{eq:Szegocond}
    \frac{1}{2\pi}\int_{-\pi}^{\pi}\log(\det M(\cos\theta))\dd \theta>-\infty, 
\end{equation}
there exists an analytic matrix valued function $D :\mathbb{C}\setminus[-1,1]\mapsto \mathbb{C}^{r\times r}$ that is invertible for every $z\in\mathbb{C}\setminus[-1,1]$, with continuous boundary values $D_{\pm}(x)$ on $(-1, 1)$ that satisfies
\begin{equation}\label{eq:WDD}
M(x) 
= 
D_-(x)D_-(x)^\ast
=
D_+(x)D_+(x)^\ast.
\end{equation}
Also, the limit
\begin{equation}\label{eq:defDinfty}
D(\infty) := \lim_{z\to\infty}D(z)
\end{equation}
exists and is an invertible matrix.
\end{proposition}

\begin{remark}
\normalfont
If $M(x)$ is real symmetric for $x\in(-1,1)$ then the matrix Szeg\H{o} function verifies
\begin{equation}
D(\overline{z})
=
\overline{D(z)},\qquad z\in\mathbb{C}\setminus[-1,1].
\end{equation}
In this situation, we have $D_{+}(x)=\overline{D_{-}(x)}$, and \eqref{eq:WDD} becomes
\begin{equation}\label{eq:WDD_real}
M(x) 
= 
D_+(x)D_-(x)^T
=
D_-(x)D_+(x)^T.
\end{equation}
\end{remark}

The existence of this matrix factorization relies on classical results, see for example the papers of Wiener and Masani \cite[Theorem 7.13]{WM1957} or Youla and Kazanjian \cite{YK1978}, and it is usually stated on the unit circle: under the condition \eqref{eq:Szegocond}, 
we have
\begin{equation}
    M\left(\frac{z+z^{-1}}{2}\right) 
    =
    G(z)G(z)^\ast, \qquad |z|=1,
\end{equation}
where $G(z)$ is analytic and invertible in $|z|<1$. Then 
    \begin{equation}\label{eq:DG}
        D(z)=G\left(\frac{1}{\varphi(z)}\right),\qquad 
        \varphi(z)=z+(z^2-1)^{1/2},
    \end{equation}
    where $\varphi(z)$ maps conformally $\mathbb{C}\setminus[-1,1]$ onto $\mathbb{C}\setminus\mathbb{D}$, is a matrix Szeg\H{o} function for $M(z)$.

The matrix Szeg\H{o} function is used in the global parametrix for MVOPs on $[-1,1]$ in \cite{DKR2023}. In that setting, $D(z)$ is constructed with a suitable factorization of the weight, that comes from representation theory. In our current situation, if $M(x)=Q(x)Q(x)^\ast$, we can only claim existence of such a factorization, but if $M(x)=\ee^{Ax}\ee^{A^\ast x}$, with $A$ a nilpotent matrix (so $M(x)$ is in fact a matrix polynomial), then a constructive algorithm is proposed in \cite{ephremidze_2014}. Building on these ideas, we have the following result:

\begin{proposition}
\normalfont
    Let $M(x)=\ee^{Ax}\ee^{A^\ast x}=G_0(x)G_0(x)^\ast$, with $A$ an $r\times r$ nilpotent matrix as in \eqref{eq:A}, then $D(z)$ given by \eqref{eq:DG} is a matrix Szeg\H{o} function for $M(z)$, where 
    \begin{equation}
    G(z)
    =
    G_0(z)
\textrm{diag}(z^{r-1},z^{r-2},\ldots,1)
\prod_{j=1}^{r-1}
\left(\prod_{k=1}^{r-j} U_k^{(j)}
\textrm{diag}(1,\ldots, z^{-1},\ldots,1)
\right).
    \end{equation}
    Here $U_k^{(j)}$ are constant unitary matrices, and $z^{-1}$ appears in the $j$-th position in the product.
    %
    
\end{proposition}

\begin{remark}
This factorization is clearly not unique, since we can consider $\widetilde{G}(z)=G(z) U$, with $U$ for example a constant unitary matrix, and the factorization still works.
\end{remark}

\begin{proof}
\normalfont
With the change of variable $x=\frac{1}{2}(z+z^{-1})$,
$M(z)$ is a Laurent polynomial in the variable $z$. We write
\begin{equation}
    M(z)=G_0(z)G_0(z)^\ast, \qquad G_0(z)=\ee^{Ax(z)}. 
\end{equation}
Because of the structure of the matrix $A$, the $j$-th column of $G_0(z)$ will have a pole at the origin of order $r-j$, for $j=1,2,\ldots r$.

We can remove these poles of $G_0(z)$ with right multiplication by a diagonal matrix with suitable positive powers of $z$:
\[
\widetilde{G}_0(z)
=
G_0(z)\textrm{diag}(z^{r-1},z^{r-2},\ldots,1).
\]

The resulting matrix $\widetilde{G}_0(z)$ is analytic in $\mathbb{D}$, but it becomes singular at the origin, since 
$\det \widetilde{G}_0(z)=z^{\frac{r(r-1)}{2}}\det G_0(z)=z^{\frac{r(r-1)}{2}}$.

In the next step, we remove zeros of the determinant (one by one) as follows: we multiply on the right by unitary matrices and then diagonal matrices with negative powers of $z$. The first column would be fixed with
\[
G_1(z)
:=
G_0(z)\textrm{diag}(z^{r-1},z^{r-2},\ldots,1)
\left(\prod_{k=1}^{r-1} U_k^{(1)}
\textrm{diag}(z^{-1},1,\ldots,1)
\right)
\]
The matrix $G_1(z)$ remains analytic in $\mathbb{D}$ if we choose the unitary matrices $U_k^{(1)}$ in such a way that at each step the product has the first column equal to $0$ when $z=0$, so that $z=0$ is a root of all the entries in that column (this happens $r-1$ times), and then the product with $z^{-1}$ on the right does not introduce any singularity.

Once we have used $r-1$ zeros of $\det \widetilde{G}_0(z)$ in this way, we move to the second column, and we build the following matrix:
\[
G_2(z)
:=
G_1(z)
\left(\prod_{k=1}^{r-2} U_k^{(2)}
\textrm{diag}(1,z^{-1},\ldots,1)
\right)
\]

We continue in the same fashion until we get to the last but one row, and as a result we obtain 
\[
G_{r-1}(z)
:=
G_0(z)
\textrm{diag}(z^{r-1},z^{r-2},\ldots,1)
\prod_{j=1}^{r-1}
\left(\prod_{k=1}^{r-j} U_k^{(j)}
\textrm{diag}(1,\ldots,z^{-1},\ldots,1)
\right),
\]
where $z^{-1}$ appears in the $j$-th position.

This matrix $G_{r-1}(z)$ is analytic and non-singular in $\mathbb{D}$, because
\[
\det G_{r-1}(z)
=
\pm \det G_{0}(z)\cdot z^{\frac{r(r-1)}{2}}
\prod_{j=1}^{r-1} z^{-(r-j)}=\pm 1.
\]

Furthermore, we can check that $G_{r-1}(z)G_{r-1}(z)^\ast=G_{0}(z)G_{0}(z)^\ast$, so the process keeps the original factorization. Therefore, we can take $G(z)=G_{r-1}(z)$ as the final matrix Szeg\H{o} function.

\end{proof}

\begin{example}
\normalfont
In the $2\times 2$ case ($r=2$), we have
\begin{equation}
    G_0(z)
    =
    \begin{pmatrix}
    1 & 0\\
    \frac{a}{2}(z+z^{-1}) & 1
    \end{pmatrix},
\end{equation}
and then we build
    \begin{equation}
    \widetilde{G}_0(z)
    =
    G_0(z) \begin{pmatrix} z & 0\\ 0 & 1\end{pmatrix}
    =
    \begin{pmatrix}
    z & 0\\ \frac{a}{2}(z^2+1) & 1
    \end{pmatrix}.
\end{equation}
We have
\[
G_1(z)
=
G_0(z)\textrm{diag}(z,1)
U_1^{(1)}\textrm{diag}(z^{-1},1)
\]
We choose the unitary factor $U_1^{(1)}$ in such a way that
\[
G_0(z)\textrm{diag}(z,1)\Big\vert_{z=0}
U_1^{(1)}
=
\begin{pmatrix}
    0 & \ast\\
    0 & \ast
\end{pmatrix}
\Rightarrow
\begin{pmatrix}
    0 & 0\\
    \frac{a}{2} & 1
\end{pmatrix}
U_1^{(1)}
=
\begin{pmatrix}
    0 & \ast\\
    0 & \ast
\end{pmatrix}
\]
We can pick (for example)
\[
U_1^{(1)}
=
\frac{1}{\sqrt{1+\frac{a^2}{4}}}
\begin{pmatrix}
1 & \frac{a}{2}\\
-\frac{a}{2} & 1
\end{pmatrix}
=
\frac{1}{\sqrt{4+a^2}}
\begin{pmatrix}
2 & a\\
-a & 2
\end{pmatrix}.
\]
As a result, we have
\begin{equation}
G(z)
=
\frac{1}{2\sqrt{a^2+4}}
\begin{pmatrix}
    4 & 2az\\
    2az & 4+a^2(1+z^2)
\end{pmatrix}.
\end{equation}
We can verify that $D(z)=G\left(\frac{1}{\varphi(z)}\right)$ satisfies all the requirements for the Szeg\H{o} function, using the fact that $\varphi_+(x)\varphi_-(x)=1$ for $x\in(-1,1)$. Also, we have
\begin{equation}
D(\infty)
=
\begin{pmatrix}
\frac{2}{\sqrt{a^2+4}} & 0\\
0 & \frac{\sqrt{a^2+4}}{2}
\end{pmatrix}.
\end{equation}
\end{example}

\begin{example}
\normalfont

In the case $r=3$, we have
\begin{equation}
G_0(z)
=
\begin{pmatrix}
    1 & 0 & 0\\
    \frac{a}{2}(z+z^{-1}) & 1 & 0\\
    \frac{a^2}{8}(z+z^{-1})^2 & \frac{a}{2}(z+z^{-1}) & 1 
\end{pmatrix}.
\end{equation}
Then, in the first step we obtain
\begin{equation}
\widetilde{G}_0(z)
=
G_0(z) 
\textrm{diag}(z^2,z,1)
=
\begin{pmatrix}
    z^2 & 0 & 0\\ 
    \frac{az}{2}(z^2+1) & z & 0\\
    \frac{a^2}{8}(z^2+1)^2 & \frac{a}{2}(z^2+1) & 1
    \end{pmatrix}.
\end{equation}

In this case, we have
\[
G_1(z)
=
G_0(z)\textrm{diag}(z^2,z,1)
U_1^{(1)}\textrm{diag}(z^{-1},1,1)
U_2^{(1)}\textrm{diag}(z^{-1},1,1).
\]
We choose the first unitary factor $U_1^{(1)}$ in such a way that
\[
G_0(z)\textrm{diag}(z^2,z,1)\Big\vert_{z=0}
U_1^{(1)}
=
\begin{pmatrix}
    0 & \ast &\ast \\
    0 & \ast & \ast\\
    0 & \ast & \ast
\end{pmatrix}
\Rightarrow
\begin{pmatrix}
    0 & 0 & 0\\
    0 & 0 & 0\\
    \frac{a^2}{8} & \frac{a}{2} & 1
\end{pmatrix}
U_1^{(1)}
=
\begin{pmatrix}
    0 & \ast &\ast \\
    0 & \ast & \ast\\
    0 & \ast & \ast
\end{pmatrix}.
\]
We have great freedom here, for example we can try
\[
U_1^{(1)}
=
\frac{1}{\sqrt{16+a^2}}
\begin{pmatrix}
    4 & a & 0 \\
    -a & 4 & 0\\
    0 & 0 & \sqrt{16+a^2}
\end{pmatrix}.
\]
If we take this choice, the next step is
\[
\begin{aligned}
G_0(z)\textrm{diag}(z^2,z,1)U_1^{(1)}\textrm{diag}(z^{-1},1,1)\Big\vert_{z=0}
U_2^{(1)}
&=
\begin{pmatrix}
    0 & 0 & 0\\
    \frac{a}{\sqrt{16+a^2}} & 0 & 0\\
    0 & \frac{a\sqrt{16+a^2}}{8} & 1
\end{pmatrix}
U_2^{(1)}\\
&=
\begin{pmatrix}
    0 & \ast &\ast \\
    0 & \ast & \ast\\
    0 & \ast & \ast
\end{pmatrix}.
\end{aligned}
\]
We can take
\[
U_2^{(1)}
=
\frac{1}{8+a^2}
\begin{pmatrix}
    0 & 0 & 8+a^2 \\
    8 & a\sqrt{16+a^2} & 0\\
    -a\sqrt{16+a^2} & 8 & 0
\end{pmatrix}.
\]
Finally, we have
\[
G_1(z)\Big\vert_{z=0}
U_1^{(2)}
=
\begin{pmatrix}
    \ast & 0 &\ast \\
    \ast & 0 & \ast\\
    \ast & 0 & \ast
\end{pmatrix},
\]
which leads to
\[
\begin{pmatrix}
    0 & 0 & 0\\
    \frac{4}{\sqrt{16+a^2}} & 0 & \frac{a}{\sqrt{16+a^2}}\\
    0 & \frac{8+a^2}{8} & 0
\end{pmatrix}
U_1^{(2)}
=
\begin{pmatrix}
    \ast & 0 &\ast \\
    \ast & 0 & \ast\\
    \ast & 0 & \ast
\end{pmatrix}.
\]
We can take
\[
U_1^{(2)}
=
\frac{1}{\sqrt{16+a^2}}
\begin{pmatrix}
    4 & a & 0 \\
    0 & 0 & \sqrt{16+a^2} \\
    a & -4 & 0
\end{pmatrix}.
\]
Writing everything together, we have
\[
\begin{aligned}
G(z)
&=
G_1(z)
U_1^{(2)}
\textrm{diag}(1,z^{-1},1)\\
&=
G_0(z)\textrm{diag}(z^2,z,1)U_1^{(1)}\textrm{diag}(z^{-1},1,1)U_2^{(1)}\textrm{diag}(z^{-1},1,1)
U_1^{(2)}
\textrm{diag}(1,z^{-1},1),  
\end{aligned}
\]
and direct calculation gives
\[
G(z)
=
\frac{1}{a^2+8}
\begin{pmatrix}
    4az & -8 & a^2z^2\\
    8+a^2(1+2z^2) & -4az &\frac{az}{2}(8+a^2(1+z^2))\\
    \frac{az}{2}(8+a^2(1+z^2)) & 
    -a^2z^2 & \frac{(8+a^2(1+z^2))^2}{8}
\end{pmatrix}.
\]
We can make $G(z)$ symmetric by considering $\widetilde{G}(z)=G(z)\begin{pmatrix} 0 & 1 & 0\\ -1 & 0& 0\\ 0 & 0 & 1\end{pmatrix}$.

Again, we can verify that $D(z)=\widetilde{G}\left(\frac{1}{\varphi(z)}\right)$ satisfies all the requirements for the Szeg\H{o} function. Also, we have
\begin{equation}
D(\infty)
=
\frac{1}{a^2+8}
\begin{pmatrix}
8 & 0 & 0\\
0 & a^2+8 & 0\\
0 & 0 & \frac{(a^2+8)^2}{8}
\end{pmatrix}.
\end{equation}

\end{example}

\begin{remark}
\normalfont
     In the cases $r=2$ and $r=3$ we can write 
\[
\prod_{j=1}^{r-1}
\left(\prod_{k=1}^{r-j} U_k^{(j)}
\textrm{diag}(1,\ldots,z^{-1},\ldots,1)
\right)
=
U
\textrm{diag}(z^{-(r-1)},z^{-(r-2)},\ldots,1),
\]
with a constant unitary matrix $U$. Again, $z^{-1}$ appears in the $j$-th position on the left hand side. In the $2\times 2$ and $3\times 3$ cases, we have
\begin{equation}
    U
    =
    \frac{1}{\sqrt{a^2+4}}
    \begin{pmatrix}
        2 & a\\ -a & 2
    \end{pmatrix},\qquad
U
=
\frac{1}{a^2+8}
\begin{pmatrix}
8 & 4a & a^2\\
-4a & 8-a^2 & 4a\\
a^2 & -4a & 8
\end{pmatrix}.
\end{equation}
This result is much more elegant, and it certainly suggests a deeper structure,  but we could not find a general proof of such property of swapping the order of the matrices in the factorization for $r\geq 4$. It is worth noticing, though, that guessing the form of such a global matrix $U$ is hard, whereas the long product can be constructed with relatively easy unitary factors.
\end{remark}

\section{Main results}
For simplicity, we study the diagonal case $n=N$, however since the dependence on $N$ is restricted to the scalar part of the weight,
we can write 
\begin{equation}\label{eq:VnN}
    \
    \ee^{-N v(z)}M(z)
    =
    \ee^{-n v_t(z)}M(z),\qquad\textrm{where}\quad 
    t=\frac{n}{N},\quad v_t(z)=\frac{1}{t}v(z).
\end{equation}


It follows from results of Kuijlaars and McLaughlin \cite[Theorem 1.3 (iii)]{KMcL2000} 
that, since $v(z)$ is analytic and $t_0=1$ is a regular value for $v(z)$ (by assumption), then 
$t\in(1-\varepsilon,1+\varepsilon)$ is a regular value too, for some $\varepsilon>0$. Also, in that neighborhood of $t_0=1$, the support of the equilibrium measure depends on $t$ but it remains one interval $[a_t,b_t]$, with $a_t$ real analytic and decreasing function of $t$, and $b_t$ real analytic and increasing function of $t$. Consequently, the asymptotic analysis will be valid for 
\begin{equation}\label{eq:nN}
    1-\varepsilon<\frac{n}{N}<1+\varepsilon
\end{equation}
for some $\varepsilon>0$. This observation is relevant for instance if one wants to use the asymptotic information together with string equations for the recurrence coefficients, see for instance  \cite{DER2021,GdIM2012}, in order to obtain more detailed information on the asymptotic expansions. This is a very advantageous approach in the scalar case, but it is unclear at this stage if these string equations can be used (even with symbolic calculations) in the matrix setting.

Our first result gives strong asymptotics for MVOPs as $N\to\infty$, for $z$ in different regions of the complex plane:
\begin{theorem}\label{th:MVOPs}
\normalfont 
Let $P_{N,N}(z)$ be the monic MVOPs with respect to the exponential weight \eqref{eq:Wx}. Then, 
as $N\to\infty$, we have the following asymptotic expansions:
\begin{enumerate}
    \item[(i)] Uniformly for $z$ in compact subsets of $\mathbb{C}\setminus[-1,1]$, (outer asymptotics), we have 
\begin{equation}\label{eq:asymp_outer}
    c^{-N}\ee^{-Ng(z)}
    P_{N,N}(cz+d)
    \sim
    \frac{\varphi(z)^{1/2}}{\sqrt{2}(z^2-1)^{1/4}}D(\infty)
    \left[I_r+\sum_{k=1}^{\infty}\frac{\Pi_k(z)}{N^k}\right]D(z)^{-1},
\end{equation}
where each coefficient $\Pi_k(z)$ is an analytic function of $z$ in this domain. Here $D(z)$ is the matrix Szeg\H{o} function corresponding to the weight $M(x)$ in \eqref{eq:MQQ}, and $D(\infty)$ is given by \eqref{eq:defDinfty}. The parameters $c,d$ are given in \eqref{eq:ctdt}, and
\begin{equation}\label{eq:varphi}
\varphi(z)
= 
z +(z^2-1)^{1/2},
\end{equation}
which is analytic in $\mathbb{C}\setminus [-1,1]$, is a conformal map from this domain onto the exterior of the unit circle. The function $g(z)$ is defined in \eqref{eq:gz}.

\item[(ii)] For $x$ in compact subsets of $(-1,1)$ (inner asymptotics), we have 
\begin{equation}\label{eq:asymp_inner}
    c^{-N} \ee^{-\frac{N}{2}(v(x)+\ell)}
    P_{N,N}(cx+d)
    =
    \frac{\sqrt{2}}{(1-x^2)^{1/4}} 
    D(\infty)
\left[\textrm{Re}
\left(
\ee^{-\ii\psi(x)}D_+(x)^{-1}\right)+\mathcal{O}(N^{-1})\right],
\end{equation}
where the phase function is 
\begin{equation}\label{eq:psi}
\psi(x)=-\frac{\ii N}{2}\phi_{+}(x)+\frac{1}{2}\arcsin(x).
\end{equation}
The function $D_+(x)$ is the boundary limit of $D(z)$ from the upper half plane, and $\ell$ is the Euler-Lagrange constant in \eqref{eq:variational}.

\item[(iii)] For $x\in(1-\delta,1)$, we have the following Airy--type asymptotics:
\begin{equation}\label{eq:asymp_edge}
    \begin{aligned}
        &N^{-\frac{1}{6}}
        c^{-N}\ee^{-\frac{N}{2}(v(x)+\ell)}
        P_{N,N}(cx+d)\\
    &=
    \frac{|f(x)|^{1/4} \sqrt{2\pi}}{(1-x^2)^{1/4}}\,\textrm{Ai}(N^{2/3}f(x)) 
    D(\infty)
\,\textrm{Re}
\left(
\ee^{\ii\left(\frac{\pi}{4}-\frac{1}{2}\arcsin(x)\right)}D_+(x)^{-1}\right)
+\mathcal{O}(N^{-1/3}),
    \end{aligned}
\end{equation}
where $f(z)$ is a conformal map in a neighborhood of $z=1$ defined in terms of $\phi(z)$, see \eqref{eq:fz}, $f(x)$ is its  value on the interval $(1-\delta,1)$ and  $\ell$ is the Euler-Lagrange constant in \eqref{eq:variational}. 

\end{enumerate}
\end{theorem}

\begin{remark}\label{rem:zeros}
\normalfont
From the previous results, we can obtain information regarding the zeros of MVOPs, that are understood as zeros of $\det\, P_{N,N}(z)$. Bearing in mind that the Szeg\H{o} function is invertible for $z\in\mathbb{C}\setminus[-1,1]$, it follows from Theorem \ref{th:MVOPs} and Hurwitz's theorem (see e.g. \cite[\S 6.4]{Simon2a}) that as $N\to\infty$, zeros of $\det\, P_{N,N}(z)$ cannot accumulate on any set of $\mathbb{C}\setminus[-1,1]$.
If we take determinants in \eqref{eq:asymp_inner}, we obtain
\begin{equation}
    \det\, P_{N,N}(cx+d)
    =
    F_N(x)
    \det\left(
    \textrm{Re}\left(\ee^{-\ii\psi(x)}D_+(x)^{-1}\right)\right)\left(1+\mathcal{O}(N^{-1})\right),
\end{equation}
where $F_N(x)\neq 0$ in the interval $(-1,1)$. As a consequence, zeros of MVOPs tend to zeros of $\textrm{Re}\left(\ee^{-\ii\psi(x)}D_+(x)^{-1}\right)$ as $N\to\infty$. If the Szeg\H{o} function is available explicitly (we include examples in Section \ref{sec:conclusions}), then we can give more detailed information about the limit distribution of zeros inside the interval $(-1,1)$.
\end{remark}

For the next result, we need the function 
$L(z)=D(z)^{-1}Q(z)Q(\overline{z})^\ast D(\overline{z})^{-\ast}$, see \eqref{eq:Lz}, which is analytic in $\mathbb{C}\setminus[-1,1]$ and admits local expansions around the points $z=\pm 1$:
\begin{equation}
L(z)
=
\begin{cases}
I_r+L_1 (z-1)^{1/2}+\mathcal{O}(z-1),&\qquad z\to 1\\
I_r+L_{-1} (z+1)^{1/2}+\mathcal{O}(z+1),&\qquad z\to -1,
\end{cases}
\end{equation}
where $L_{\pm 1}$ are some $r\times r$ matrices. 

As an example, in the case $r=2$, if $Q(x)=\ee^{Ax}$, with $A=\begin{pmatrix} 0 & 0 \\ a & 0\end{pmatrix}$, we obtain
\begin{equation}
L_1
=
\frac{2\sqrt{2}a}{a^2+4}
\begin{pmatrix}
-a & 2\\
2 & a
\end{pmatrix},\qquad 
L_{-1}
=
\frac{2\sqrt{2}a}{a^2+4}
\begin{pmatrix}
a & 2\\
2 & -a
\end{pmatrix}.
\end{equation}

These matrices $L_1$ and $L_{-1}$ appear in the next result, which gives the asymptotic expansions for recurrence coefficients:

\begin{theorem}\label{th:BnCn}
\normalfont 
The recurrence coefficients in \eqref{eq:TTRRN} satisfy
\begin{equation}
\begin{aligned}
B_{N,N}\sim B^{(0)}+\sum_{j=1}^{\infty} \frac{B^{(j)}}{N^j},\qquad
C_{N,N}\sim C^{(0)}+\sum_{j=1}^{\infty} \frac{C^{(j)}}{N^j},
\end{aligned}
\end{equation}
as $N\to\infty$. The first coefficients are
\begin{equation}\label{eq:asympBnCn_theorem}
\begin{aligned}
B^{(0)}&=d I_r, \qquad \\
C^{(0)}&=\frac{c^2}{4}I_r\qquad
C^{(1)}=
\frac{\sqrt{2} c^2}{8}
D(\infty)
\left(\frac{L_1}{h(1)}-\frac{L_{-1}}{h(-1)}\right)
D(\infty)^{-1},
\end{aligned}
\end{equation}
in terms of $D(\infty)$ again and the constants $c$ and $d$ given in \eqref{eq:ctdt}. The function $h(x)$ is directly related to the density of the equilibrium measure, see \eqref{eq:defpsit}.
\end{theorem}

\begin{remark}
    \normalfont
    In the scalar case, Bleher and Its \cite[Section 5]{BI2005} show that if $v(x)$ is one-cut regular (meaning that the corresponding equilibrium measure is supported on a single interval and has square root vanishing at the endpoints), then the recurrence coefficients (with a suitable shift in the argument of $B_{N,N}$) admit asymptotic expansions in powers of $N^{-2}$. This  is a refinement of the general asymptotic result that originates from the steepest descent analysis, and it is proved using the string equations, which are nonlinear equations for these recurrence coefficients that can be obtained in different ways. In the matrix case, string equations are available (see for instance \cite{DER2021}), 
and their analysis is in principle possible, but we do not address it in this paper. Formula \eqref{eq:asympBnCn_theorem}, however, indicates that it is no longer true in general that odd terms in the expansion vanish in the matrix case. 
\end{remark}
\begin{remark}
\normalfont
    This asymptotic result is consistent with the explicit formulas obtained in \cite{DG2005} for $2\times 2$ Hermite MVOPs. We note however that in that reference (and later ones), the weight function is presented with the matrix $A^\ast$ instead of $A$ (in our notation), and this requires modification of the matrix Szeg\H{o} function. Additionally, it is necessary to scale $a\mapsto a N^{-1/2}$, and the recurrence coefficients $B_n\mapsto B_n N^{-1/2}$ and $C_n\mapsto C_n N^{-1}$.
\end{remark}

The steepest descent analysis allows us to give the asymptotic behavior of the matrix norm $\mathcal{H}_N$ as well:
\begin{theorem}\label{th:Hn}
\normalfont 
The norms of the MVOPs have the following asymptotic behavior:

\begin{equation}\label{eq:asympHNN}
\begin{aligned}
\mathcal{H}_{N,N}
\sim
\pi
c^{2N+1} \ee^{N\ell}
D(\infty)
\left[
I_r
+
\sum_{j=1}^{\infty}
\frac{\mathcal{H}^{(j)}}{N^j}
\right]
D(\infty)^\ast.
\end{aligned}
\end{equation}
where the first term is 
\begin{equation}\label{eq:H1}
\begin{aligned}
\mathcal{H}^{(1)}
&=
\frac{4h(1)-3h'(1)}{24h(1)^2}I_r+\frac{L_1(2\sqrt{2}I_r+L_1)}{8h(1)}\\
&+
\frac{4h(-1)+3h'(-1)}{24h(-1)^2}I_r
-\frac{L_{-1}(2\sqrt{2}I_r-L_{-1})}{8h(-1)}.
    \end{aligned}
\end{equation}
Here the matrix $D(\infty)$, the function $h(x)$, the constant $\ell$ and the matrices $L_{\pm 1}$ are the same as in the previous theorem.
\end{theorem}

\begin{corollary}
    \normalfont
    The recurrence coefficients for $A_{N,N}$ and $\widehat{B}_{N,N}$ for orthonormal polynomials, see \eqref{eq:TTRR_orthonorm}, satisfy
    \begin{equation}\label{eq:asympANBN}
A_{N,N}
=
\frac{c}{2}I_r+\mathcal{O}(N^{-1}), \qquad     \widehat{B}_{N,N}
=
d I_r+\mathcal{O}(N^{-1}),
    \end{equation}
as $N\to\infty$, where $c$ and $d$ are the constants in \eqref{eq:ctdt}.
\end{corollary}
\begin{proof}
\normalfont
From Remark \ref{rem:orthonorm}, taking $n=N$, we observe the following:
\begin{equation}\label{eq:CNAN}
C_{N,N}
=
\kappa_{N-1,N}^{-1}A_{N,N}A_{N,N}^\ast \kappa_{N-1,N}
=
\kappa_{N,N}^{-1}A_{N,N}^\ast A_{N,N}\kappa_{N,N}.
\end{equation}
From Theorem \ref{th:Hn}, writing 
$\Lambda_N=\pi
c^{2N+1} \ee^{N\ell}
D(\infty)D(\infty)^\ast$, we obtain 
\[
\mathcal{H}_{N,N}=(\kappa_{N,N}^\ast \kappa_{N,N})^{-1}
=
\Lambda_N \left(I_r+\mathcal{O}(N^{-1})\right), 
\]
as $N\to\infty$, since $D(\infty)$ is independent of $N$. So,
\begin{equation}\label{eq:kappaNkappaN}
\kappa_{N,N}^\ast \kappa_{N,N}
=
\left(I_r+\mathcal{O}(N^{-1})\right)\Lambda_N ^{-1}.
\end{equation}
Up to unitary factors on the left, which must remain bounded in $N$, we can take 
\begin{equation}\label{eq:kappaN}
\kappa_{N,N}=\Lambda_N^{-1/2} \left(I_r+\mathcal{O}(N^{-1})\right).
\end{equation}
With this choice, it follows that
\[
\begin{aligned}
\kappa_{N,N}^\ast\kappa_{N,N}
&=
\left(I_r+\mathcal{O}(N^{-1})\right)\Lambda_N^{-1} \left(I_r+\mathcal{O}(N^{-1})\right)\\
&=    
\left(I_r+\mathcal{O}(N^{-1})\right)\Lambda_N^{-1} 
+
\left(I_r+\mathcal{O}(N^{-1})\right))\Lambda_N^{-1} \mathcal{O}(N^{-1}),
\end{aligned}
\]
as $N\to\infty$. Since the dependence on $N$ in $\Lambda_N$ appears only in scalar factors, we have 
\[
\Lambda_N^{-1} \mathcal{O}(N^{-1})
=
\Lambda_N^{-1} \mathcal{O}(N^{-1})\Lambda_N\Lambda_N^{-1}
=
\mathcal{O}(N^{-1})\Lambda_N^{-1},
\]
and we recover \eqref{eq:kappaNkappaN}. As a consequence, combining \eqref{eq:CNAN} with \eqref{eq:kappaN}, we obtain
\[
\begin{aligned}
    A_{N,N}^\ast A_{N,N} 
    &=
    \kappa_{N,N} C_{N,N} \kappa_{N,N}^{-1}\\
    &=
    \Lambda_N^{-1/2}\left(I_r+\mathcal{O}(N^{-1})\right)
    \left(\frac{c^2}{4}I_r+\mathcal{O}(N^{-1})\right)
    \left(I_r+\mathcal{O}(N^{-1})\right)\Lambda_N^{1/2}\\
    &=
    \frac{c^2}{4}I_r+\mathcal{O}(N^{-1}).
     \end{aligned}
\]
From this result, we obtain
\[
A_{N,N}
=
\frac{c}{2}I_r+\mathcal{O}(N^{-1}).
\]
Similarly, the other recurrence coefficient satisfies
\[
\begin{aligned}
    \widehat{B}_{N,N}
=
\kappa_{N,N}B_{N,N}\kappa_{N,N}^{-1}
&=
\Lambda_N^{-1/2}\left(I_r+\mathcal{O}(N^{-1})\right)
    \left(d I_r+\mathcal{O}(N^{-1})\right)
    \left(I_r+\mathcal{O}(N^{-1})\right)\Lambda_N^{1/2}\\
    &=
    d I_r+\mathcal{O}(N^{-1}),
\end{aligned}
\]
which concludes the proof.
\end{proof}

This asymptotic behavior of the recurrence coefficients 
$A_{N,N}$ and $\widehat{B}_{N,N}$ 
for orthonormal MVOPs can be combined with the result by Delvaux and Dette in \cite[Section 3]{DD_2012} in order to calculate the limiting zero counting measure of $\det P_{N,N}(z)$. In the notation of that reference, if we write
\begin{equation}
    \lim_{n/N\to s} A_{n,N}:=A_s, \qquad
    \lim_{n/N\to s} \widehat{B}_{n,N}:=B_s, 
\end{equation}
then from the previous discussion, we have limit values when $s=1$, namely $A_1=\frac{c}{2}I_r$ and $B_1=dI_r$. We consider the following equation:
\begin{equation}
    0
    =
    f_1(z,x):=
    \det(A_1^\ast z+B_1+A_1 z^{-1}-xI_r)
\end{equation}
In our case, since the limits $A_1$ and $B_1$ are multiples of the identity matrix, this becomes:
\begin{equation}
    0
    =
    \det\left(\frac{cz}{2}I_r +d I_r+\frac{c}{2}I_r-xI_r\right)
    =
    \prod_{k=1}^r \left(\frac{cz}{2}+d +\frac{c}{2z}-x\right)
    =
    \left(\frac{cz}{2}+d +\frac{c}{2z}-x\right)^r.
\end{equation}
The solutions $z=z(x)$ are 
\begin{equation}\label{eq:zx}
    z_{\pm}(x)=\frac{2(x-d)\pm 2\sqrt{(x-d)^2-c^2}}{c},
\end{equation}
each with multiplicity $r$. In this case, the set that attracts the zeros of MVOPs (via eigenvalues of block Toeplitz matrices) is 
\begin{equation}
    \Gamma_0
    =
    \{x\in\mathbb{C}:|z_-(x)|=|z_+(x)|\}.
\end{equation}
Such equality holds if $(x-d)^2-c^2\leq 0$, and this corresponds to $x\in[a,b]$, using \eqref{eq:ctdt}. This is consistent with the results that we obtain from the asymptotic expansion of the MVOPs, see Remark \ref{rem:zeros}.

More generally, if we consider the limit $n/N \to s$, for $s>0$, then the roots in \eqref{eq:zx}, and consequently the set $\Gamma_0$ will depend on $s$. We write $\Gamma_0(s)$ to indicate this dependence.

If we consider the normalized zero counting measure
\begin{equation}
    \nu_{n,N}=\frac{1}{rn}\sum_{k=1}^{rn}\delta(x-x_{k,N}),
\end{equation}
where $x_{k,N}$ are the zeros of $\det P_{N,N}(x)$, counted with multiplicities, then Theorem 3.2 in \cite[Section 3]{DD_2012} gives a general formula for the logarithmic potential of the limit measure $\mu_{0,1}$ as $n/N\to 1$:
\begin{equation}
    \int \log|x-t|\dd\mu_{0,1}(t)
    =\frac{1}{r}
    \int_0^1 \log|z_1(x,u)\cdots z_r(x,u)|\dd u+C, \qquad x\in\mathbb{C}\setminus \displaystyle\cup_{0\leq u\leq 1}\Gamma_0(x,u).
\end{equation}
For computational purposes, this formula requires (explicit) knowledge of how the roots (and the support of the equilibrium measure) change with the parameter $s$, which is complicated for a generic polynomial potential $v(x)$ in \eqref{eq:Wx}.

\section{Steepest descent analysis}

\subsection{First transformation. Scaling}
Our starting point is the Riemann--Hilbert problem for MVOPs in Section \ref{sec:RHid}. 
Using the constants $c$ and $d$ given in \eqref{eq:ctdt} and bearing in mind Remark \ref{rem:sigma3}, we make the first transformation, which is just a shift/scaling of the interval:
\begin{equation}\label{eq:VY}
V(z)=
c^{-N\sigma_3}Y(cz+d).
\end{equation}
This matrix satisfies the following Riemann--Hilbert problem:
\begin{enumerate}
\item $V(z)$ is analytic in $\mathbb{C}\setminus \mathbb{R}$.
\item For $x\in\mathbb{R}$, the matrix $V(z)$ has boundary values $V_{\pm}(x)=\lim_{\varepsilon\to 0}V(x\pm\ii\varepsilon)$, that satisfy the following jump condition:
\begin{equation}
V_+(x)
=
V_-(x)
\begin{pmatrix}
I_r & \ee^{-N v(cx+d)}M(x)\\
0_r & I_r
\end{pmatrix},
\qquad 
M(x)=Q(cx+d)Q(cx+d)^{\ast}.
\end{equation}
\item As $z\to\infty$, we have the asymptotic behavior
\begin{equation}\label{eq:asympU}
V(z)
=
\left(I_{2r}+\frac{V^{(1)}}{z}+\mathcal{O}(z^{-2})\right) z^{N\sigma_3}
\end{equation}
\end{enumerate}

By direct calculation, we can see that the first correction matrix $V^{(1)}$ can be written in terms of $Y^{(1)}$ as
\begin{equation}\label{eq:U1Y1}
V^{(1)}
=
c^{-1}
\left[c^{-N\sigma_3}
Y^{(1)}
c^{N\sigma_3}
+
dN\sigma_3\right].
\end{equation}

\subsection{Second transformation. Equilibrium problem and $g$ function}
With the assumptions that we have set, the support of the equilibrium measure for $v(x)$ is a single interval $[a,b]$ which is mapped to $[-1,1]$ by the change of variable $z\mapsto cz+d$, recall \eqref{eq:ctdt}. If $v(x)$ is a polynomial of degree $2m$, see \eqref{eq:Wx}, then the density of the equilibrium measure is given by
\begin{equation}\label{eq:defpsit}
  \eta(x)=\frac{1}{2\pi}h(x)\sqrt{1-x^2}, \qquad x\in[-1,1].
\end{equation}
where $h(x)$ is a polynomial of degree $2m-2$, given by
\begin{equation}
\frac{\frac{\dd}{\dd x}v(cx+d)}{(x^2-1)^{1/2}}
=
h(x)+\mathcal{O}(x^{-1}),\qquad x\to\infty.
\end{equation}
We assume that $h(x)>0$ for $x\in(-1,1)$, to avoid critical cases arising from vanishing of the density of the equilibrium measure.


The $g$-function that we use to normalize the Riemann--Hilbert problem at infinity is 
\begin{equation}\label{eq:gz}
g(z)
=
\int_{-1}^{1} \log(z-s)\eta(s)\dd s,
\end{equation}
 which is analytic in $\mathbb{C}\setminus(-\infty,1]$. This function satisfies the Euler--Lagrange variational conditions
\begin{equation}\label{eq:variational}
g_{+}(x)+g_{-}(x)-v(x)-\ell
\begin{cases}
    =0, &\qquad  \textrm{a.e.}\,\, x\in (-1,1),\\
    < 0, & \qquad x\in \mathbb{R}\setminus[-1,1],
\end{cases}
\end{equation}
for some constant $\ell$. We also define
\begin{equation}\label{eq:phit}
\phi(z)
=
\int_{1}^z h(s)(s^2-1)^{1/2}\dd s, \qquad 
\widetilde{\phi}(z)
=
\int_{-1}^z h(s)(s^2-1)^{1/2}\dd s.
\end{equation}
The function $\phi(z)$ is analytic in $\mathbb{C}\setminus(-\infty,1]$, and $\widetilde{\phi}(z)$ is analytic in $\mathbb{C}\setminus[-1,\infty)$, and the contours of integration do not cross the corresponding cuts. 


 Alternatively, we can define
 \begin{equation}\label{eq:phitphi}
     \widetilde{\phi}(z)
     =
     \phi(z)\pm 2\pi\ii, 
     \qquad \pm \textrm{Im}\, z>0.
 \end{equation}

We make the following transformation:
\begin{equation}\label{eq:UtoT}
T(z)
=
\ee^{-\frac{N}{2}\ell\sigma_3}
V(z)
\ee^{-N\left(g(z)-\frac{\ell}{2}\right)\sigma_3}.
\end{equation}
Then this new matrix satisfies the following RHP:
\begin{enumerate}
\item $T(z)$ is analytic in $\mathbb{C}\setminus\mathbb{R}$.
\item For $x\in\mathbb{R}$, we have the jump
\begin{equation}
T_+(x)
=
T_-(x)
\begin{cases}
\begin{pmatrix}
I_r & \ee^{-N\tilde{\phi}(x)}M(x)\\
0_r & I_r
\end{pmatrix},& \qquad x<-1,\\
\begin{pmatrix}
\ee^{N\phi_{+}(x)}I_r & M(x)\\ 
0_r & \ee^{-N\phi_{+}(x)}I_r\\ 
\end{pmatrix},& \qquad -1<x<1,\\
\begin{pmatrix}
I_r & \ee^{-N\phi(x)}M(x)\\
0_r & I_r
\end{pmatrix},& \qquad x>1.
\end{cases}
\end{equation}

\item As $z\to\infty$, we have the asymptotics
\begin{equation}
T(z)
=
I_{2r}
+\frac{T^{(1)}}{z}+\mathcal{O}(z^{-2}).
\end{equation}
\end{enumerate}

\subsection{Third transformation. Opening of a lens}
In the next step, we factor the jump matrix on the interval $(-1,1)$:
\begin{multline}
\begin{pmatrix}
\ee^{N\phi_{+}(x)}I_r & M(x)\\
0_r & \ee^{-N\phi_{+}(x)}I_r
\end{pmatrix}\\
=
\begin{pmatrix}
I_r & 0_r  \\ 
 \ee^{N\phi_{-}(x)}M(x)^{-1}& I_r
\end{pmatrix}
\begin{pmatrix}
0_r & M(x)\\ 
-M(x)^{-1} & 0_r
\end{pmatrix}
\begin{pmatrix}
I_r & 0_r  \\ 
\ee^{N\phi_{+}(x)}M(x)^{-1}& I_r
\end{pmatrix}.
\end{multline}

\begin{figure}    \centerline{\includegraphics[scale=1]{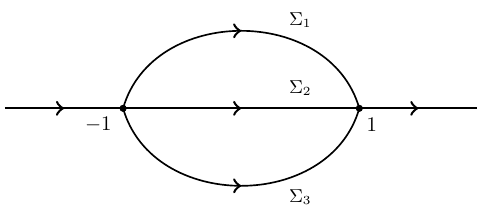}}
    \caption{Lens for the $T\mapsto S$ transformation.}
    \label{fig:lens}
\end{figure}

We open a lens-shaped domain around the interval $(-1,1)$, with contours $\Sigma_1\cup \Sigma_2\cup\Sigma_3$, labelled downwards, see Figure \ref{fig:lens}, and we define the following matrix:
\begin{equation}
S(z)
=
T(z)
\begin{cases}
\begin{pmatrix}
I_r & 0_r\\
-\ee^{N\phi(z)}M(z)^{-1} & I_r
\end{pmatrix},& \qquad \textrm{in the upper part of the lens},\\
\begin{pmatrix}
I_r & 0_r\\
\ee^{N\phi(z)}M(z)^{-1} & I_r
\end{pmatrix},& \qquad \textrm{in the lower part of the lens}.
\end{cases}
\end{equation}

This new matrix satisfies the following Riemann--Hilbert problem: 
\begin{enumerate}
\item $S(z)$ is analytic in $\mathbb{C}\setminus \left(\mathbb{R}\cup\Sigma_1\cup\Sigma_3\right)$. 
\item On these contours, it has the jumps
\begin{equation}
S_+(z)
=
S_-(z)
\begin{cases}
\begin{pmatrix}
I_r & \ee^{-N\widetilde{\phi}(z)} M(z)\\
0_r & I_r
\end{pmatrix}, & \qquad z<-1,\\
\begin{pmatrix}
I_r& 0_r\\
\ee^{N\phi(z)}M(z)^{-1} & I_r 
\end{pmatrix}, & \qquad z\in\Sigma_1\cup\Sigma_3,\\
\begin{pmatrix}
0_r & M(z)\\ 
-M(z)^{-1} & 0_r
\end{pmatrix}, & \qquad -1<z<1,\\
\begin{pmatrix}
I_r & \ee^{-N\phi(z)}M(z) \\ 
0_r & I_r
\end{pmatrix}, & \qquad z>1.
\end{cases}
\end{equation}
\item As $z\to\infty$, we have $S(z)=I_{2r}+\mathcal{O}(z^{-1})$.
\end{enumerate}

In this extension of the problem to the complex plane, we note that the factorization of the weight that we have to use is $M(z)=Q(z)Q(
\overline{z})^\ast$, to keep analyticity properties.

\subsection{Global parametrix and matrix Szeg\H{o} function}
To construct the global parametrix, we ignore all the jumps of the matrix $S(z)$ that become close to identity as $N\to\infty$, and we solve the following problem: 
\begin{enumerate}
\item $P^{(\infty)}(z)$ is analytic in $\mathbb{C}\setminus [-1,1]$.
\item On this interval it satisfies 
\begin{equation}\label{eq:jumpPinfty}
P^{(\infty)}_+(x)
=
P^{(\infty)}_-(x)
\begin{pmatrix}
0_r & M(x) \\ 
-M(x)^{-1} & 0_r
\end{pmatrix},  \qquad x\in (-1,1).
\end{equation}
\item As $z\to\infty$, we have $P^{(\infty)}(z)=I_{2r}+\mathcal{O}(z^{-1})$.
\end{enumerate}

We solve this global problem using a suitable matrix Szeg\H{o} function $D(z)$, see Proposition \ref{prop:SzegoD}. 
Once we have this matrix Szeg\H{o} function, we construct
\begin{equation}\label{eq:Pinfty}
P^{(\infty)}(z)
=
\begin{pmatrix}
D(\infty) & 0_r\\
0_r & D(\infty)^{-\ast}
\end{pmatrix}
\widetilde{P}^{(\infty)}(z)
\begin{pmatrix}
D(z)^{-1} & 0_r\\
0_r & D(\overline{z})^\ast
\end{pmatrix},
\end{equation}
and it follows that $\widetilde{P}^{(\infty)}(z)$ solves the same global problem but with jump matrix $\begin{pmatrix} 0_r & I_r\\ -I_r & 0_r\end{pmatrix}$. 

By diagonalising the jump for $\widetilde{P}^{(\infty)}(z)$, we can build the solution explicitly, with blocks that are multiples of the identity matrix:
\begin{equation}\label{eq:Ptinfty}
\widetilde{P}^{(\infty)}(z)
=
\frac{1}{2}
\begin{pmatrix}
(\beta(z)+\beta(z)^{-1})I_r & -\ii(\beta(z)-\beta(z)^{-1})I_r\\
\ii(\beta(z)-\beta(z)^{-1})I_r & (\beta(z)+\beta(z)^{-1})I_r
\end{pmatrix}, \qquad \beta(z)=\left(\frac{z-1}{z+1}\right)^{1/4},
\end{equation}
where the root has a cut on $[-1,1]$.

In order to find a factorization with the right properties, we map the interval $[-1,1]$ onto the unit circle:
\begin{equation}\label{eq:xtoz}
     x
     =
     \frac{z+z^{-1}}{2}, \qquad |z|=1.
\end{equation}
 In general, we have the factorization
\begin{equation}
M(z)
=
G(z)G(z)^\ast, \qquad |z|=1,
\end{equation}
where the matrix $G(z)$ is analytic and invertible in the interior of the unit disc $|z|<1$, and we use the function $G(1/\overline{z})^\ast$ when we evaluate it in the exterior of the unit disc.


\subsection{Local parametrix at $z=1$}
We consider a disc $D_{\delta}(1)$ of fixed radius $\delta>0$ around the endpoint $z=1$. We construct a matrix $P(z)$ that has the same jumps as $S(z)$ inside the disc, so 
we have
\begin{enumerate}
\item $P(z)$ is analytic in $D_{\delta}(1)\setminus \left(\mathbb{R}\cup\Sigma_1\cup\Sigma_3\right)$, see Figure \ref{fig:D1}. 
\item On these contours, it has the jumps
\begin{equation}
P_+(z)
=
P_-(z)
\begin{cases}
\begin{pmatrix}
I_r & 0_r\\
\ee^{N\phi(z)}M(z)^{-1} & I_r 
\end{pmatrix}, & \qquad z\in D_{\delta}(1)\cap (\Sigma_1\cup\Sigma_3),\\
\begin{pmatrix}
0_r & M(z)\\
-M(z)^{-1} & 0_r
\end{pmatrix}, & \qquad 1-\delta<z<1,\\
\begin{pmatrix}
I_r & \ee^{-N\phi(z)}M(z)\\ 
0 & I_r
\end{pmatrix}, & \qquad 1<z<1+\delta.
\end{cases}
\end{equation}
\item Uniformly for $z\in\partial D_{\delta}(1)$, we have the following matching condition with the global parametrix:
\begin{equation}
P(z)=\left(I_{2r}+\mathcal{O}(N^{-1})\right)P^{(\infty)}(z), \qquad N\to\infty.
\end{equation}
\item As $z\to 1$, $P(z)$ has the same behavior as $S(z)$, in the sense that the product $S(z)P^{-1}
(z)$ remains bounded as $z\to 1$.
\end{enumerate}

\begin{figure}
    \centerline{\includegraphics[scale=1]{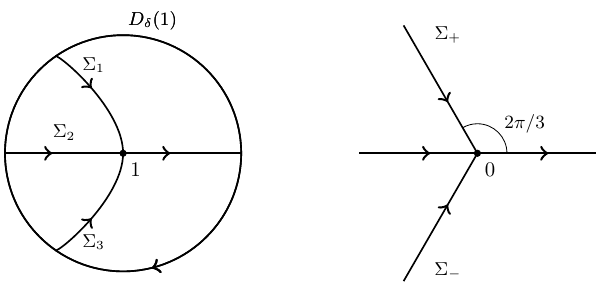}}
    \caption{Disc $D_{\delta}(1)$ for the local parametrix (left) and contours for the RH problem for $\Psi(\zeta)$ (right).}
    \label{fig:D1}
\end{figure}

The first step is to make constant jumps. We write
\begin{equation}\label{eq:P1P}
P(z)
=
E_N(z)P^{(1)}(z)
\begin{pmatrix}
\ee^{\frac{N}{2}\phi(z)} Q(z)^{-1} & 0_r\\
0_r & \ee^{-\frac{N}{2}\phi(z)} Q(\overline{z})^{\ast}
\end{pmatrix},
\end{equation}
where $E_N(z)$ is a suitable analytic prefactor, to be determined later. 

This new matrix $P^{(1)}(z)$ satisfies the following Riemann--Hilbert problem:
\begin{enumerate}
\item $P^{(1)}(z)$ is analytic in $D_{\delta}(1)\setminus \left(\mathbb{R}\cup\Sigma_1\cup\Sigma_3\right)$. 
\item On these contours, it has the jumps
\begin{equation}
P^{(1)}_+(z)
=
P^{(1)}_-(z)
\begin{cases}
\begin{pmatrix}
I_r & 0_r\\
I_r & I_r
\end{pmatrix}, & \qquad z\in D_{\delta}(1)\cap (\Sigma_1\cup\Sigma_3),\\
\begin{pmatrix}
0_r & I_r\\
-I_r & 0_r
\end{pmatrix}, & \qquad 1-\delta<z<1,\\
\begin{pmatrix}
I_r & I_r\\
0_r & I_r
\end{pmatrix}, & \qquad 1<z<1+\delta.
\end{cases}
\end{equation}
\item Uniformly for $z\in\partial D_{\delta}(1)$, we have the following matching condition with the global parametrix:
\begin{equation}
P^{(1)}(z)=\left(I_{2r}+\mathcal{O}(N^{-1})\right)P^{(\infty)}(z)
\begin{pmatrix}
\ee^{\frac{N}{2}\phi(z)} Q(z)^{-1} & 0_r\\
0_r & \ee^{-\frac{N}{2}\phi(z)}Q(\overline{z})^{\ast}
\end{pmatrix}.
\end{equation}
\end{enumerate}

We need a conformal map $f(z)$ from a neighborhood of $z=1$ onto a neighborhood of $\zeta=0$. We take 
\begin{equation}\label{eq:fz}
\frac{\phi(z)}{2}=\frac{2}{3} f(z)^{3/2} 
\Rightarrow
f(z)=\left(\frac{3}{4}\phi(z)\right)^{2/3}.
\end{equation} 
From the definition of $\phi(z)$, we have the local behavior
\begin{equation}\label{eq:localft}
f(z)
=2^{-1/3}h(1)^{2/3}(z-1)\left(1+\mathcal{O}(z-1)\right).
\end{equation}

We take advantage of the freedom in choosing the contours in the $z$ plane in such a way that $\Sigma_2\cap D_{\delta}(1)$ is mapped to the negative real axis and the interval $(1,1+\delta)$ is mapped to the positive real axis. Thus, $f(z)$ is real valued on the real axis.

We set up an auxiliary RH problem, in a new variable $\zeta$: we seek $\Psi(\zeta)$ such that
\begin{enumerate}
\item $\Psi(z)$ is analytic in $\mathbb{C}\setminus\left(\mathbb{R}\cup\Sigma_{\pm}\right)$, where $\Sigma_{\pm}=(0,\infty)\ee^{\pm2\pi\ii/3}$, with the real axis oriented from left to right and $\Sigma_{\pm}$ towards the origin, see Figure \ref{fig:D1}. 
\item On these contours, it has the jumps
\begin{equation}
\Psi_+(\zeta)
=
\Psi_-(\zeta)
\begin{cases}
\begin{pmatrix}
I_r & 0_r\\
I_r & I_r
\end{pmatrix}, & \qquad \zeta\in \Sigma_{\pm},\\
\begin{pmatrix}
0_r & I_r\\
-I_r & 0_r
\end{pmatrix}, & \qquad  \zeta\in\mathbb{R}^{-},\\
\begin{pmatrix}
I_r & I_r\\
0_r & I_r
\end{pmatrix}, & \qquad  \zeta\in\mathbb{R}^+.
\end{cases}
\end{equation}
\item As $\zeta\to\infty$, we have the asymptotic behavior
\begin{equation}
\Psi(\zeta)
=
\zeta^{-\sigma_3/4}
\frac{1}{\sqrt{2}}
\begin{pmatrix}
I_r & \ii I_r\\
\ii I_r & I_r
\end{pmatrix}
(I_{2r} +\mathcal{O}(\zeta^{-3/2}))
\ee^{-\frac{2}{3}\zeta^{3/2}\sigma_3}.
\end{equation}
\item As $\zeta\to 0$, $\Psi(\zeta)$ remains bounded.
\end{enumerate}

We can solve the local problem with a suitable combination of Airy functions (we understand the following formula with $r\times r$ blocks, and we skip all the $I_r$ factors throughout for brevity:
\begin{align}\label{eq: Airy parametrix}
					\Psi(\zeta) =\sqrt{2\pi}
					\begin{cases}
					\begin{pmatrix}
					y_0(\zeta) & -y_2(\zeta) \\
					-\ii y_0'(\zeta) & \ii y_2'(\zeta) 
					\end{pmatrix}, \qquad & \arg \zeta \in \left(0, \frac{2\pi}{3}\right), \\
					\begin{pmatrix}
					-y_1(\zeta) & -y_2(\zeta) \\
					\ii y_1'(\zeta) & \ii y_2'(\zeta) 
					\end{pmatrix}, \qquad & \arg \zeta \in \left(\frac{2\pi}{3},\pi\right),\\
					\begin{pmatrix}
					-y_2(\zeta) & y_1(\zeta) \\
					\ii y_2'(\zeta) & -\ii y_1'(\zeta) 
					\end{pmatrix}, \qquad & \arg \zeta \in \left(-\pi, -\frac{2\pi}{3}\right),\\
					\begin{pmatrix}
					y_0(\zeta) & y_1(\zeta) \\
					-\ii y_0'(\zeta) & -\ii y_1'(\zeta) 
					\end{pmatrix}, \qquad & \arg \zeta \in \left(-\frac{2\pi}{3},0\right).
					\end{cases}
				\end{align}
				where 
				\begin{equation}\label{eq: Airy functions}
					y_0(\zeta) := \text{Ai}(\zeta), \qquad y_1(\zeta):=\omega\text{Ai}(\omega \zeta), \qquad y_2(\zeta):=\omega^2 \text{Ai}(\omega^2 \zeta),
				\end{equation}
where $\text{Ai}$ is the Airy function and $\omega:=\exp\left(2\pi\ii/3\right)$. We use the identity $y_0(\zeta)+y_1(\zeta)+y_2(\zeta)=0$.

We can actually write more detailed asymptotics for the local parametrix: \begin{equation}\label{eq:asympA_full}
\Psi(\zeta)
\sim 
\frac{1}{\sqrt{2}}
\zeta^{-\sigma_3/4}
    \begin{pmatrix}
    I_r & \ii I_r\\
    \ii I_r & I_r
    \end{pmatrix}
    \left(I_{2r}+\sum_{k=1}^{\infty}\frac{\Psi_k}{\zeta^{3k/2}}\right)
\ee^{-\frac{2}{3}\zeta^{3/2}\sigma_3},
\end{equation}
for $|\arg\,\zeta|<\pi$, where the coefficients for $k\geq 1$ are
\begin{equation}\label{eq:Psik}
    \Psi_k
    =
    \frac{(3/2)^k}{2}
    \begin{pmatrix}
    (-1)^k (u_k+v_k)I_r & \ii(u_k-v_k)I_r\\
   (-1)^{k+1} \ii(u_k-v_k)I_r & (u_k+v_k) I_r
    \end{pmatrix}
\end{equation}
with $u_0=v_0=1$ and 
\begin{equation}\label{eq:uk}
u_{k}=\frac{(2k+1)(2k+3)(2k+5)\cdots(6k-1)}{216^{k}k!}
=
\frac{\Gamma\left(3k+\frac{1}{2}\right)}{54^k k!\Gamma\left(k+\frac{1}{2}\right)}
,\qquad
v_k=\frac{6k+1}{1-6k}u_k,
\end{equation}
for $k\geq 1$, see \cite[9.7.2]{NIST:DLMF}.

If we set $\widetilde{P}(z)=\Psi(N^{2/3} f(z))$, then the local parametrix is
\begin{equation}\label{eq:P1}
P(z)
=
E_N(z)
\Psi(N^{2/3} f(z))
\begin{pmatrix}
\ee^{\frac{N}{2}\phi(z)}Q(z)^{-1} & 0_r\\
0_r & \ee^{-\frac{N}{2}\phi(z)}Q(\overline{z})^{\ast}
\end{pmatrix}.
\end{equation}
Since we have the matching condition $P(z)=(I+\mathcal{O}(N^{-1}))P^{(\infty)}(z)$ on the boundary of the disc $D_{\delta}(1)$, using \eqref{eq:asympA_full}, we need 
\begin{equation}\label{eq:EN}
\begin{aligned}
E_N(z)
&=
P^{(\infty)}(z)
\begin{pmatrix}
Q(z) & 0_r\\
0_r & Q(\overline{z})^{-\ast}
\end{pmatrix}
%
\frac{1}{\sqrt{2}}
\begin{pmatrix}
I_r & -\ii I_r\\
-\ii I_r & I_r
\end{pmatrix}
(N^{2/3}f(z))^{\sigma_3/4}.
\end{aligned}
\end{equation}

\begin{lemma}
    \normalfont
    The matrix $E_N(z)$ is an analytic matrix valued function in $D_{\delta}(1)$.
\end{lemma}
\begin{proof}
\normalfont
    We first check that 
$E_{N+}(x)=E_{N-}(x)$ for $x\in(-1,1)$, which makes use of the fact that $f(z)=f_0(z-1)m(z)$, where $f_0>0$, $m(z)$ is analytic and $h(1)\neq 0$, and then
\begin{equation}
    \left(f(z)^{\alpha}\right)_{\pm}=|f(x)|^{\alpha}\ee^{\pm\alpha\pi\ii}, \qquad x\in(1-\delta,1).
\end{equation}
Then, direct calculation shows that $E_{N-}(x)^{-1}E_{N+}(x)=I_r$ for $x\in(1-\delta,1)$. This implies that any singularity of $E_N(z)$ at $z=1$ must be isolated. 

Now we need to consider the limit behavior as $z\to 1$, to rule out an essential singularity or a pole of $E_N(z)$ at that point: we write
\begin{equation}
\begin{aligned}
E_N(z)
&=
\begin{pmatrix}
D(\infty) & 0_r\\
0_r & D(\infty)^{-\ast}
\end{pmatrix}
\widetilde{P}^{(\infty)}(z)
\begin{pmatrix}
D(z)^{-1}Q(z) & 0_r\\
0_r & D(\overline{z})^\ast Q(\overline{z})^{-\ast}
\end{pmatrix}\\
&\times \frac{1}{\sqrt{2}}
\begin{pmatrix}
I_r & -\ii I_r\\
-\ii I_r & I_r
\end{pmatrix}
(N^{2/3}f(z))^{\sigma_3/4}.
\end{aligned}
\end{equation}
As in \cite[\S 3.5.7]{DKR2023}, we write this as follows:
\begin{equation}
    E_N(z)
    =
    \frac{1}{\sqrt{2}}
    \begin{pmatrix}
        D(\infty) & 0_r\\
        0_r & D(\infty)^{-\ast}
    \end{pmatrix}
    \begin{pmatrix}
        I_r & \ii I_r\\
        \ii I_r & I_r
    \end{pmatrix}
    \left(
    N^{2/3}E^{(1)}(z)+N^{-2/3}E^{(2)}(z)
    \right),
\end{equation}
where $E^{(1)}(z)$ and $E^{(2)}(z)$ are $2r\times r$ matrices, independent of $N$, that can be written as follows:
\begin{equation}
\begin{aligned}
    E^{(1)}(z)
    =
    \begin{pmatrix}
    \beta(z)f(z)^{1/4}\left(D(z)^{-1}Q(z)-D(\overline{z})^\ast Q(\overline{z})^{-\ast}\right)\\
    -\ii \beta(z)^{-1}f(z)^{1/4}
\left(D(z)^{-1}Q(z)+D(\overline{z})^\ast Q(\overline{z})^{-\ast}\right)\\
    \end{pmatrix},\\
    E^{(2)}(z)
    =
    \begin{pmatrix}
    -\ii\beta(z)f(z)^{-1/4}\left(D(z)^{-1}
    Q(z)+
    D(\overline{z})^\ast Q(\overline{z})^{-\ast}\right)\\
    \beta(z)^{-1}f(z)^{-1/4}
\left(-D(z)^{-1}Q(z)+D(\overline{z})^\ast Q(\overline{z})^{-\ast}\right)
    \end{pmatrix}.
\end{aligned}
\end{equation}

Bearing in mind that $\beta(z),f(z)^{1/4}=\mathcal{O}\left((z-1)^{1/4}\right)$ as $z\to 1$, we define
\begin{equation}\label{eq:Omega12}
    \begin{aligned}
        \Omega_1(z)
        &=   D(z)^{-1}Q(z)+D(\overline{z})^{\ast}Q(\overline{z})^{-\ast},\\
        \Omega_2(z)
        &=
        (z-1)^{-1/2}\left(D(z)^{-1}Q(z)-D(\overline{z})^\ast Q(\overline{z})^{-\ast}\right)
    \end{aligned}
\end{equation}

We want to show that $\Omega_1(z)$ and $\Omega_2(z)$ have removable singularities at $z=1$. To this end, we approach this endpoint along the interval $(1-\delta,1)$ on the real axis.

We consider first the factor $D(z)^{-1}Q(z)$. If we take boundary values on the interval $(1-\delta,1)$, we can write $D_+(x)=M(x)^{1/2} Z(x)^\ast$, where $Z(x)$ is unitary and $M(x)^{1/2}$ denotes the Hermitian positive definite root of $M(x)$, that we recall that is Hermitian positive definite on the interval). Using this, and the spectral decomposition $M(x)=U(x)\Lambda(x)U(x)^\ast$, with $U(x)$ unitary, we have
\[
\begin{aligned}
\left(D(z)^{-1}Q(z)\right)_+
=
D_+(x)^{-1}Q(x)
=
Z(x) M(x)^{-1/2}Q(x)
&=
Z(x) (U(x)\Lambda(x)U(x)^{\ast})^{-1/2}Q(x)\\
&=
Z(x) U(x)\Lambda(x)^{-1/2}U(x)^{\ast}Q(x).
\end{aligned}
\]
 It follows that the product $D_+(x)^{-1}Q(x)$ remains bounded as $x\to 1^-$. The same result holds for the $-$ boundary value. This is a direct consequence of Rellich's theorem, which states that spectral quantities (eigenvalues and eigenvectors) can be arranged in such a way that they are analytic functions of $z$ in a neighborhood of $z=1$, see \cite[Theorem 1.4.4]{BS2015_4} or \cite{Wimmer1986}. Moreover, we have
 \begin{equation}
D_+(x)^{-1}Q(x)\left(D_+(x)^{-1}Q(x)\right)^\ast
=
D_+(x)^{-1}Q(x)Q(x)^{\ast} D_+(x)^{-\ast}
=I_r,
\end{equation}
since $Q(x)Q(x)^\ast=M(x)=D_+(x)D_+(x)^\ast$, because of the properties of the matrix Szeg\H{o} function. Therefore, $D_+(x)^{-1}Q(x)$ is actually a unitary matrix for $x\in(1-\delta,1)$, and the same is true with the $-$ boundary values. Since the limit as $x\to 1^-$ exists, this limit value $\lim_{x\to 1^-}D_+(x)^{-1}Q(x)$ remains unitary. 

Regarding the second term, $D(\overline{z})^{\ast}Q(\overline{z})^{-\ast}$, we consider
\[
\begin{aligned}
    \left(D(\overline{z})^{\ast}Q(\overline{z})^{-\ast}\right)_-
    =
    D_+(x)^{\ast} Q(x)^{-\ast}
    =
    \left(D_+(x)^{-1}Q(x)\right)^{-\ast},
\end{aligned}
\]
so this factor remains bounded as $x\to 1^-$ as well.
Then, it follows from \eqref{eq:Omega12} that $\Omega_1(x)=\mathcal{O}(1)$ and 
$\Omega_2(x)=\mathcal{O}((x-1)^{-1/2})$ as $x\to 1^-$. This, together with the fact that any singularity at $z=1$ should be isolated, implies that these two factors $\Omega_1(z)$ and $\Omega_2(z)$ have removable singularities at that point. The same is true for $E_N(z)$ and this matrix is actually analytic in a neighborhood of $z=1$.
\end{proof}



\begin{remark}\label{rem:Dunit}
It follows that we can write a local expansion
\begin{equation}\label{eq:D1Q}
D(z)^{-1}Q(z)=D_{1,0}+D_{1,1}(z-1)^{1/2}+\mathcal{O}(z-1),\qquad z\to 1,
\end{equation} 
where $D_{1,0}$ is a unitary matrix. Note that $D(z)^{-1}Q(z)$ is not unitary in general for $z$ away from the real axis. 
\end{remark}

\subsection{Local parametrix at $z=-1$}
A similar construction can be done in the disc $D_{\delta}(-1)$. We consider a disc $D_{\delta}(-1)$ of fixed radius $\delta>0$ around the endpoint $z=-1$, and we construct a matrix $\widetilde{P}(z)$ that has the same jumps as $S(z)$. Therefore, the matrix $\widetilde{P}(z)$ solves the following RH problem:
\begin{enumerate}
\item $\widetilde{P}(z)$ is analytic in $D_{\delta}(-1)\setminus \left(\mathbb{R}\cup\Sigma_1\cup\Sigma_3\right)$. 
\item On these contours, it has the jumps
\begin{equation}
\widetilde{P}_+(z)
=
\widetilde{P}_-(z)
\begin{cases}
\begin{pmatrix}
I_r & 0_r\\
\ee^{N\widetilde{\phi}(z)}M(z)^{-1} & I_r 
\end{pmatrix}, & \qquad z\in D_{\delta}(-1)\cap (\Sigma_1\cup\Sigma_3),\\
\begin{pmatrix}
0_r & M(z)\\
-M(z)^{-1} & 0_r
\end{pmatrix}, & \qquad -1<z<-1+\delta,\\
\begin{pmatrix}
I_r & \ee^{-N\widetilde{\phi}(z)}M(z)\\ 
0 & I_r
\end{pmatrix}, & \qquad -1-\delta<z<-1.
\end{cases}
\end{equation}
We have written $\widetilde{\phi}(z)$ instead of $\phi(z)$ in the jumps on account of \eqref{eq:phitphi}.

\item Uniformly for $z\in\partial D_{\delta}(-1)$, we have the following matching condition with the global parametrix:
\begin{equation}
\widetilde{P}(z)=\left(I+\mathcal{O}(N^{-1})\right)P^{(\infty)}(z).
\end{equation}
\end{enumerate}

As before, we build this local parametrix as 
\begin{equation}
\widetilde{P}(z)
=
\widetilde{E}_N(z)
\Psi(N^{2/3} \widetilde{f}(z))
\begin{pmatrix}
\ee^{\frac{N}{2}\widetilde{\phi}(z)}
Q(z)^{-1} & 0_r\\
0_r & \ee^{-\frac{N}{2}\widetilde{\phi}(z)} Q(\overline{z})^{\ast}.
\end{pmatrix}.
\end{equation}

The conformal map now is 
\begin{equation}\label{eq:ftz}
\widetilde{f}(z)
=
\left(\frac{3}{4}\widetilde{\phi}(z)\right)^{2/3},
\end{equation}
with local behavior $\widetilde{f}(z)=\widetilde{f}_0(z+1)\left(1+\mathcal{O}(z+1)\right)$ as $z\to -1$, with $\widetilde{f}_0<0$. So this takes $z=-1$ to $\zeta=0$ and reverses orientation of the contours. As a consequence 
\begin{equation}
\widetilde{P}(z)
=
\widetilde{E}_N(z)
\begin{pmatrix} 
I_r & 0_r\\
0_r & -I_r
\end{pmatrix}
\Psi(N^{2/3} \widetilde{f}(z))
\begin{pmatrix} 
I_r & 0_r\\
0_r & -I_r
\end{pmatrix}
\begin{pmatrix}
\ee^{\frac{N}{2}\widetilde{\phi}(z)}
Q(z)^{-1} & 0_r\\
0_r & \ee^{-\frac{N}{2}\widetilde{\phi}(z)}Q(\overline{z})^{\ast}
\end{pmatrix}.
\end{equation}
Because of the matching on the boundary of the disc, we need 
\begin{equation}\label{eq:ENt}
    \widetilde{E}_N(z)
    =
    P^{(\infty)}(z)
    \begin{pmatrix}
Q(z) & 0_r\\
0_r & Q(\overline{z})^{\ast}
\end{pmatrix}\\
\frac{1}{\sqrt{2}}
\begin{pmatrix}
I_r & \ii I_r\\
\ii I_r & I_r
\end{pmatrix}
(N^{2/3}\widetilde{f}(z))^{\sigma_3/4}.
\end{equation}

Similarly as before, we can check that $\widetilde{E}_N(z)$ does not have a jump on the interval $(-1,-1+\delta)$, by writing
\begin{equation}
    \widetilde{E}_N(z)
    =
    \frac{1}{\sqrt{2}}
    \begin{pmatrix}
        D(\infty) & 0_r\\
        0_r & D(\infty)^{-\ast}
    \end{pmatrix}
    \begin{pmatrix}
        I_r & \ii I_r\\
        \ii I_r & I_r
    \end{pmatrix}
    \left(
    N^{2/3}\widetilde{E}^{(1)}(z)+N^{-2/3}\widetilde{E}^{(2)}(z)
    \right),
\end{equation}
where $\widetilde{E}^{(1)}(z)$ and $\widetilde{E}^{(2)}(z)$ are $2r\times r$ matrices, independent of $n$:
\begin{equation}
\begin{aligned}
    \widetilde{E}^{(1)}(z)
    =
    \begin{pmatrix}
\beta(z)\widetilde{f}(z)^{1/4}\left(D(z)^{-1}Q(z)+
D(\overline{z})^\ast Q(\overline{z})^{-\ast}\right)\\
    -\ii \beta(z)^{-1}\widetilde{f}(z)^{1/4}
\left(D(z)^{-1}Q(z)-D(\overline{z})^\ast Q(\overline{z})^{-\ast}\right)\\
    \end{pmatrix},\\
    \widetilde{E}^{(2)}(z)
    =
    \begin{pmatrix}
    \ii\beta(z)\widetilde{f}(z)^{-1/4}\left(D(z)^{-1}Q(z)-
    D(\overline{z})^\ast Q(\overline{z})^{-\ast}\right)\\
    \beta(z)^{-1}\widetilde{f}(z)^{-1/4}
\left(-D(z)^{-1}Q(z)-D(\overline{z})^\ast Q(\overline{z})^{-\ast}\right)
    \end{pmatrix}.
\end{aligned}
\end{equation}

Since $\beta(z)^{-1},\widetilde{f}(z)^{1/4}=\mathcal{O}\left((z+1)^{1/4}\right)$ as $z\to -1$, we define
\begin{equation}\label{eq:Omegat12}
    \begin{aligned}
        \widetilde{\Omega}_1(z)
        &=
        (z+1)^{-1/2}\left(D(z)^{-1}Q(z)-D(\overline{z})^\ast Q(\overline{z})^{-\ast}\right),\\
        \widetilde{\Omega}_2(z)
        &=
        D(z)^{-1}Q(z)+D(\overline{z})^\ast Q(\overline{z})^{-\ast}.
    \end{aligned}
\end{equation}

Using an argument similar to the one before, we can show that $\widetilde{\Omega}_1(z)$ and $\widetilde{\Omega}_2(z)$ have removable singularities at $z=-1$, and hence $\widetilde{E}_N(z)$ is analytic in a neighborhood of $z=-1$. Also, if we write the local expansion
\begin{equation}
D(z)^{-1}Q(z)=D_{-1,0}+D_{-1,1}(z+1)^{1/2}+\mathcal{O}(z+1),\qquad z\to -1,
\end{equation} 
then $D_{-1,0}$ is a unitary matrix.

\subsection{Final transformation}
In the last step of the steepest descent method, we set $\Sigma=\mathbb{R}\cup \Sigma_1\cup\Sigma_3$ and we construct the following matrix:
\begin{equation}
R(z)
=
\begin{cases}
S(z)P^{(\infty)}(z)^{-1},& \qquad z\in \mathbb{C}\setminus\left(\overline{D_{\delta}(1)}\cup \overline{D_{\delta}(-1)}\cup\Sigma\right),\\
S(z)P(z)^{-1},& \qquad z \in D_{\delta}(1)\setminus\Sigma,\\
S(z)\widetilde{P}(z)^{-1},& \qquad z \in D_{\delta}(-1)\setminus\Sigma.
\end{cases}
\end{equation}

This matrix is analytic in $\mathbb{C}\setminus \Sigma_R$, see Figure \ref{fig:SigmaR}, and on this set of contours it has the following jumps:
\begin{equation}
J_R(z)
=
\begin{cases}
P^{(\infty)}(z)
\begin{pmatrix}
I_r & 0_r\\
\ee^{N\phi(z)}M(z)^{-1}  & I_r
\end{pmatrix}
P^{(\infty)}(z)^{-1},& \quad z\in \Sigma_R\setminus \left(\partial D_{\delta}(1)\cup \partial D_{\delta}(-1)\right),\\
P(z)P^{(\infty)}(z)^{-1},& \quad z\in \partial D_{\delta}(1),\\
\widetilde{P}(z)P^{(\infty)}(z)^{-1},& \quad z\in \partial D_{\delta}(-1).
\end{cases}
\end{equation}
\begin{figure}
    \centerline{\includegraphics[scale=1]{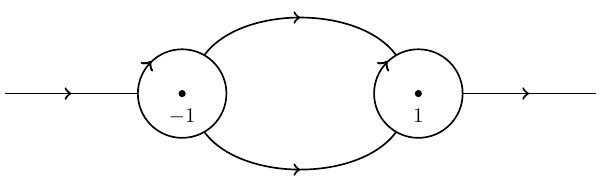}}
\caption{Final union of contours $\Sigma_R$.}
\label{fig:SigmaR}
\end{figure}

We write the jump of this final matrix as $J_R(z)=I_{2r}+\Delta(z)$, where
$\Delta(z)$ admits an asymptotic expansion in inverse powers of $N$:
\begin{equation}\label{eq:jumpDelta}
\Delta(z)
\sim
\sum_{k=1}^{\infty}\frac{\Delta_k(z)}{N^k}, \qquad N\to\infty.
\end{equation}
The coefficients $\Delta_k(z)$ are identically zero on $\Sigma_R\setminus \left(\partial D_{\delta}(1)\cup \partial D_{\delta}(-1)\right)$, because the jump on this part of the final contour is exponentially close to identity. On the boundary of the discs, $\Delta_k(z)$ can be calculated from the matching between the global and local parametrices.

\begin{lemma}
    \normalfont
    On the boundary of the discs, the coefficients $\Delta_k$ in \eqref{eq:jumpDelta} are given by 
    \begin{equation}\label{eq:Deltak}
        \Delta_k(z)
        =
        \begin{cases}
\displaystyle           
\frac{1}{f(z)^{3k/2}}
P^{(\infty)}(z)
\begin{pmatrix}
Q(z) & 0_r\\
0_r & Q(\overline{z})^{-\ast}
\end{pmatrix}
\Psi_k
\begin{pmatrix}
Q(z)^{-1} & 0_r\\
0_r & Q(\overline{z})^{\ast}
\end{pmatrix}
P^{(\infty)}(z)^{-1},\\
\displaystyle
\frac{1}{\widetilde{f}(z)^{3k/2}}
P^{(\infty)}(z)
\begin{pmatrix}
Q(z) & 0_r\\
0_r & -Q(\overline{z})^{-\ast}
\end{pmatrix}
\Psi_k
\begin{pmatrix}
Q(z)^{-1} & 0_r\\
0_r & -Q(\overline{z})^{\ast}
\end{pmatrix}
P^{(\infty)}(z)^{-1},\\
        \end{cases}
    \end{equation}
for $z\in\partial D_{\delta}(1)$ and for $z\in\partial D_{\delta}(-1)$ respectively, and for $k\geq 1$.
\end{lemma}
\begin{proof}
    \normalfont
Let us look first at the jump on $\partial D_{\delta}(1)$:
\begin{equation}
\begin{aligned}
J_R(z)
&=
P(z)P^{(\infty)}(z)^{-1}\\
&=
E_N(z)
\Psi(N^{2/3} f(z))
\begin{pmatrix}
\ee^{\frac{N}{2}\phi(z)}Q(z)^{-1} & 0_r\\
0_r & \ee^{-\frac{N}{2}\phi(z)}Q(\overline{z})^{\ast}
\end{pmatrix}
P^{(\infty)}(z)^{-1}.
\end{aligned}
\end{equation}
Using formula \eqref{eq:EN} for $E_N(z)$,the asymptotic expansion \eqref{eq:asympA_full}  and the conformal map \eqref{eq:fz}, we can write
\begin{equation}
\begin{aligned}
    J_R(z)-I_{2r}\sim
\sum_{k=1}^{\infty}
P^{(\infty)}(z)
\begin{pmatrix}
Q(z) & 0_r\\
0_r & Q(\overline{z})^{-\ast}
\end{pmatrix}
    \frac{\Psi_k}{N^k f(z)^{3k/2}}
\begin{pmatrix}
Q(z)^{-1} & 0_r\\
0_r & Q(\overline{z})^{\ast}
\end{pmatrix}
P^{(\infty)}(z)^{-1},
\end{aligned}
\end{equation}
for $z\in\partial D_{\delta}(1)$, 
where $\Psi_k$ are given in \eqref{eq:Psik} and \eqref{eq:uk}, and this leads to $\Delta_k(z)$ as in the first line of \eqref{eq:Deltak}.

For $z\in\partial D_{\delta}(-1)$, the jump is
\begin{equation}
\begin{aligned}
J_R(z)
&=
\widetilde{P}(z)P^{(\infty)}(z)^{-1}\\
&=
\widetilde{E}_N(z)
\sigma_3
\Psi(N^{2/3} \widetilde{f}(z))
\sigma_3
\begin{pmatrix}
\ee^{\frac{N}{2}\widetilde{\phi}(z)}
Q(z)^{-1} & 0\\
0 & \ee^{-\frac{N}{2}\widetilde{\phi}(z)}Q(\overline{z})^{\ast}
\end{pmatrix}
P^{(\infty)}(z)^{-1}.
\end{aligned}
\end{equation}
Using \eqref{eq:ENt} and \eqref{eq:asympA_full} again, we arrive at the second line of \eqref{eq:Deltak}.
\end{proof}

Further properties of $\Delta_k(z)$ can be obtained with a local analysis. The following result is analogous to \cite[Lemma 8.2]{KMcLVAV2004}:
\begin{lemma}\label{lem:Deltak}
\normalfont
    For $k\geq 1$, the coefficients $\Delta_k(z)$ (respectively $\widetilde{\Delta}_k(z)$) can be extended to a meromorphic function in $D_{\delta}(1)$ (respectively $D_{\delta}(-1)$), with a pole of order at most $\left[\frac{3k+1}{2}\right]$ at $z=1$ (respectively $z=-1$). 
\end{lemma}
\begin{proof}
\normalfont
    From \eqref{eq:Deltak}, we have
\begin{equation}
    \begin{aligned}
        \Delta_{k-}^{-1}(z)\Delta_{k+}(z)
        &=
        (-1)^kP_-^{(\infty)}(z)
        \begin{pmatrix}
        Q(z) & 0_r\\
        0_r & Q(\overline{z})^{-\ast}
        \end{pmatrix}\\
        &\times 
        \Psi_k^{-1}
        \begin{pmatrix}
        0_r & I_r\\
        -I_r & 0_r
        \end{pmatrix}    
        \Psi_k
        \begin{pmatrix}
        Q(z)^{-1} & 0_r\\
        0_r & Q(\overline{z})^{\ast}
        \end{pmatrix}    
        \left(P^{(\infty)}(z)\right)_+^{-1}.
    \end{aligned}
\end{equation}
The overall sign comes from the fact that $f(z)=C(z-1)+\mathcal{O}((z-1)^2)$, with $C>0$, see \eqref{eq:localft}, and then $\left(f^{3k/2}(z)\right)_-\left(f^{-3k/2}(z)\right)_+=\ee^{-\frac{3k\pi\ii}{2}}\ee^{-\frac{3k\pi\ii}{2}}=\ee^{-3k\pi\ii}=(-1)^k$. Next, from direct calculation using \eqref{eq:Psik}, we obtain
\[
\Psi_k^{-1}
        \begin{pmatrix}
        0_r & I_r\\
        -I_r & 0_r
        \end{pmatrix}    
\Psi_k
=
(-1)^k \begin{pmatrix}
    0_r & I_r\\
    -I_r & 0_r
\end{pmatrix},
\]
and then $\Delta_{k-}^{-1}(z)\Delta_{k+}(z)=I_{2r}$ by direct calculation using the jump for $P^{(\infty)}(z)$, see \eqref{eq:jumpPinfty}, and the fact that $M(z)=Q(z)Q(\overline{z})^{\ast}$. This shows that $\Delta_k(z)$ is analytic in a punctured neighborhood $D_{\delta}(1)\setminus\{1\}$.

Next, we want to show that the function $\Delta_k(z)$ can in fact be extended to a meromorphic function in a neighborhood of $z=1$. Using \eqref{eq:Pinfty}, we have
\begin{equation}
    \begin{aligned}
\Delta_k(z)&=
\frac{1}{f(z)^{3k/2}}
\begin{pmatrix}
D(\infty) & 0_r\\
0_r & D(\infty)^{-\ast}
\end{pmatrix}
\widetilde{P}^{(\infty)}(z)\\
&\times 
\begin{pmatrix}
D(z)^{-1}Q(z) & 0_r\\
0_r & D(\overline{z})^\ast Q(\overline{z})^{-\ast}
\end{pmatrix}
\Psi_k
\begin{pmatrix}
Q(z)^{-1} D(z) & 0_r\\
0_r & Q(\overline{z})^{\ast} D(\overline{z})^{-\ast}
\end{pmatrix}\\
&\times 
\widetilde{P}^{(\infty)}(z)^{-1}
\begin{pmatrix}
D(\infty)^{-1} & 0_r\\
0_r & D(\infty)^{\ast}
\end{pmatrix}.
\end{aligned}
\end{equation}

We know that $\widetilde{P}^{(\infty)}(z)=\mathcal{O}((z-1)^{-1/4})$, by $r\times r$ blocks,  directly from its definition, $f(z)=\mathcal{O}(z-1)$ as $z\to 1$, and the factor $D(z)^{-1}Q(z)$ is bounded as $z\to 1$. As a consequence, we obtain
\begin{equation}
    \Delta_k(z)
    =
    \mathcal{O}\left((z-1)^{-\frac{3k}{2}-\frac{1}{2}}\right), \qquad z\to 1.
\end{equation}
Therefore, $\Delta_k(z)$ has a pole of order at most $\left[\frac{3k+1}{2}\right]$ at $z=1$. If $k=2s-1$ is odd, we obtain a pole of order at most $3s-1$, and if $k=2s$ is even, we obtain a pole of order at most $3s$.

A similar calculation can be applied to $\widetilde{\Delta}_-^{-1}(z)\widetilde{\Delta}_+(z)$ around $z=-1$. 
\end{proof}

As a consequence of the asymptotic behavior of the jump matrix $J_R(z)$ as $N\to\infty$, standard theory gives asymptotic behavior of the matrix $R(z)$ itself. Namely, we have
\begin{equation}\label{eq:asympR}
    R(z)
    \sim I_{2r}+\sum_{k=1}^{\infty}\frac{R_k(z)}{N^k}, \qquad N\to\infty,
\end{equation}
where each coefficient $R_k(z)$ is an analytic function of $z$ in $\mathbb{C}\setminus \left(\partial D_{\delta}(1)\cup\partial D_{\delta}(-1)\right)$, and also $R_k(z)=\mathcal{O}(z^{-1})$ as $z\to\infty$.

An essential feature of the expansion \eqref{eq:asympR} is that it is uniform for $z$ near infinity. The proof of this property (see for instance \cite[Lemma 8.3]{KMcLVAV2004}) furnishes a sequence of additive Riemann--Hilbert problems for the coefficients $R_k(z)$ that can be solved sequentially. Using \eqref{eq:jumpDelta} and expanding the equation $R_+(z)=R_-(z)\left(I_{2r}+\Delta(z)\right)$ in powers of $N$ and identifying terms with equal powers, we obtain
\begin{equation}\label{eq:RHforRk}
    R_{k+}(z)
    =
    R_{k-}(z)
    +
    \sum_{j=1}^{k}
    R_{(k-j)-}(z)\Delta_j(z), \qquad z\in \partial D_{\delta}(1)\cup\partial D_{\delta}(-1),
\end{equation}
for $k\geq 1$, with $R_{0-}(z)=I_{2r}$.

From this, $R_{k}(z)$ can be calculated using the local properties of $\Delta_j(z)$ and the residue theorem.

\subsection{Calculation of $R_1(z)$}

For $k=1$, we have from \eqref{eq:RHforRk} the following Riemann--Hilbert problem:
\begin{equation}\label{eq:RHforR1}
    R_{1+}(z)=R_{1-}(z)+\Delta_1(z), 
    \qquad z\in \partial D_{\delta}(1)\cup\partial D_{\delta}(-1).
\end{equation}
This problem can be solved using Sokhotski--Plemelj formula:
\begin{equation}
    R_1(z)
    =
    \frac{1}{2\pi\ii}
    \oint_{\partial D_{\delta}(1)\cup\partial D_{\delta}(-1)}
    \frac{\Delta_1(s)}{s-z}\dd s.
\end{equation}

Since $\Psi_{2s-1,11}=-\Psi_{2s-1,22}$ and $\Psi_{2s-1,12}=\Psi_{2s-1,21}$, and the fact that $\det \widetilde{P}^{(\infty)}(z)=1$ and has blocks that are multiples of the identity (and therefore commute with other matrices), we can write the first correction matrix in the following form:
\begin{equation}\label{eq:Delta1general}
\begin{aligned}
\Delta_{1}(z)
&=
\frac{1}{f(z)^{\frac{3}{2}}}
\begin{pmatrix}
D(\infty) & 0_r\\
0_r & D(\infty)^{-\ast}
\end{pmatrix}
\widetilde{P}^{(\infty)}(z)\\
&\times
\begin{pmatrix}
    \Psi_{1,11} & \Psi_{1,12} L(z)\\
    \Psi_{1,12}L(z)^{-1}  & -\Psi_{1,11} 
\end{pmatrix}
\widetilde{P}^{(\infty)}(z)^{-1}
\begin{pmatrix}
D(\infty)^{-1} & 0_r\\
0_r & D(\infty)^\ast
\end{pmatrix},
\end{aligned}
\end{equation}
where 
\begin{equation}\label{eq:Lz}
L(z)=D(z)^{-1} Q(z)Q(\overline{z})^{\ast}
D(\overline{z})^{-\ast}.
\end{equation}

As a consequence of Lemma \ref{lem:Deltak}, the matrix $\Delta_1(z)$ has a pole of order $2$ at $z=1$. We write
\begin{equation}\label{eq:Delta1_z1}
    \Delta_1(z)
    =
    \frac{\Delta_{1,-2}}{(z-1)^2}
    +
    \frac{\Delta_{1,-1}}{z-1}
    +\mathcal{O}(1), \qquad z\to 1.
\end{equation}

Using the structure of the matrix $\widetilde{P}^{(\infty)}(z)$, and in particular the fact that its entries are multiples of the identity and therefore commute, we can write $\Delta_1(z)$ as follows:
\begin{equation}
    \begin{aligned}
        \Delta_{1}(z)
        &=
        \frac{\Psi_{1,11}}{(z^2-1)^{1/2}f(z)^{\frac{3}{2}}}
        \begin{pmatrix}
            D(\infty) & 0_r\\
            0_r & D(\infty)^{-\ast}
        \end{pmatrix}
        \begin{pmatrix}
            z I_r & -\ii I_r\\
            -\ii I_r & -z I_r    
        \end{pmatrix}
        \begin{pmatrix}
            D(\infty)^{-1} & 0_r\\
            0_r & D(\infty)^{\ast}
        \end{pmatrix}\\
        &+
        \frac{\Psi_{1,12}}{2(z^2-1)^{1/2}f(z)^{\frac{3}{2}}}\\
        &\times
        \begin{pmatrix}
            D(\infty) & 0_r\\
            0_r & D(\infty)^{-\ast}
        \end{pmatrix}
        \begin{pmatrix}
            \ii(L(z)+L(z)^{-1}) &
            \varphi(z) L(z)+\varphi(z)^{-1}L(z)^{-1}\\
            \varphi(z) L(z)^{-1}+\varphi(z)^{-1}L(z)
           &
            -\ii(L(z)+L(z)^{-1})
        \end{pmatrix}\\
        &\times
        \begin{pmatrix}
            D(\infty)^{-1} & 0_r\\
            0_r & D(\infty)^{\ast}
        \end{pmatrix}.
    \end{aligned}
\end{equation}

Here $\varphi(z)=z+(z^2-1)^{1/2}$ is the usual conformal map. As $z\to 1$, we have
\begin{equation}
\begin{aligned}
&
    \begin{pmatrix}
            \ii(L(z)+L(z)^{-1}) &
            \varphi(z) L(z)+\varphi(z)^{-1}L(z)^{-1}\\
            \varphi(z) L(z)^{-1}+\varphi(z)^{-1}L(z)
           &
            -\ii(L(z)+L(z)^{-1})
        \end{pmatrix}\\
&=
        2\begin{pmatrix}
            \ii I_r & I_r\\
            I_r & -\ii I_r
        \end{pmatrix}
+
       2\begin{pmatrix}
            0_r & I_r\\
            I_r & 0_r
        \end{pmatrix}
        (z-1)
        +
        \begin{pmatrix}
           \ii  L_1^2 & L_1(2\sqrt{2}I_r+L_1)\\
            L_1(-2\sqrt{2}I_r+L_1) & -\ii L_1^2
        \end{pmatrix}
        (z-1)\\
        &+\mathcal{O}((z-1)^2).
\end{aligned}
\end{equation}
The first two terms in this expansion will give a block scalar contribution, but the last one is genuinely matrix valued, since $L_1$ is not, in general, a multiple of the identity matrix.

We use the following expansions as $z\to 1$: 
\begin{equation}
(z^2-1)^{1/2}f(z)^{3/2}
=
h(1)(z-1)^{2}+\frac{2h(1)+3h'(1)}{5}(z-1)^{3}+\mathcal{O}((z-1)^{4}), 
\end{equation}
and 
\begin{equation}\label{eq:Lat1}
L(z)
=
I_r+L_1(z-1)^{1/2}+L_2(z-1)+\mathcal{O}((z-1)^{3/2}).
\end{equation}
In this last expansion, we have
\begin{equation} L_1=D_{1,0}D_{1,1}^\ast+D_{1,1}D_{1,0}^\ast.
\end{equation}
with $D_{1,0}$ a unitary matrix, see Remark \ref{rem:Dunit}.
Thus, we obtain
\begin{equation}\label{eq:Delta1}
    \begin{aligned}
    \Delta_{1,-2}
    &=
    \frac{5}{48h(1)}
    \begin{pmatrix}
     -I_r & \ii D(\infty)D(\infty)^\ast\\
     \ii D(\infty)^{-\ast}D(\infty)^{-1} & I_r
    \end{pmatrix},\\
    \Delta_{1,-1}
    &=
    \frac{1}{48 h(1)^2}
\begin{pmatrix}
     3(h(1)+h'(1))I_r
     &
     \ii\left(4h(1)-3h'(1)\right)D(\infty)D(\infty)^\ast\\
      \ii\left(4h(1)-3h'(1)\right)D(\infty)^{-\ast}D(\infty)^{-1}
      &
      -3(h(1)+h'(1))I_r
      \end{pmatrix}\\
&       
      +\frac{1}{16h(1)}
\begin{pmatrix}
D(\infty) & 0_r\\
0_r & D(\infty)^{-\ast}
\end{pmatrix}
      \begin{pmatrix}
       -L_1^2 
       &
       \ii L_1(2\sqrt{2}I_r+L_1)\\
       \ii L_1(-2\sqrt{2}I_r+L_1) 
       &
       L_1^2 
       \end{pmatrix}
    \begin{pmatrix}
    D(\infty)^{-1} & 0_r\\
    0_r & D(\infty)^\ast
    \end{pmatrix}
 \end{aligned}
\end{equation}

We notice an important difference with respect to the scalar case, see for example \cite[(4.11)]{KT2009}. This corresponds to $L_1=0_r$, and in that situation the matrices $\Delta_{1,-2}(z)$ and $\Delta_{1,-1}(z)$ can be written as a combination (by blocks) of the Pauli matrices $\sigma_1$ and $\sigma_3$. If $L_1\neq 0_r$, in a genuine matrix case, this is no longer the case, and the structure of the coefficients is broken.

Regarding the behavior in $D_{\delta}(-1)$, we have 
\begin{equation}
(z^2-1)^{1/2}\widetilde{f}(z)^{3/2}
=
-h(-1)(z+1)^{2}+\frac{2h(-1)-3h'(-1)}{5}(z+1)^{3}+\mathcal{O}((z+1)^{4}), 
\end{equation}
as $z\to -1$, and 
\begin{equation}\label{eq:Latm1}
L(z)
=
I_r+L_{-1}(z+1)^{1/2}+L_{-2}(z+1)+\mathcal{O}((z+1)^{3/2}),
\end{equation}
with 
\begin{equation} L_{-1}=D_{-1,0}D_{-1,1}^\ast+D_{-1,1}D_{-1,0}^\ast.
\end{equation}

As a consequence, we obtain
\begin{equation}\label{eq:Delta1_zm1}
    \Delta_1(z)
    =
    \frac{\widetilde{\Delta}_{1,-2}}{(z+1)^2}
    +
    \frac{\widetilde{\Delta}_{1,-1}}{z+1}
    +\mathcal{O}(1), \qquad z\to -1,
\end{equation}
with matrices 
\begin{equation}\label{eq:Deltat1}
    \begin{aligned}
    \widetilde{\Delta}_{1,-2}
    &=
    \frac{5}{48h(-1)}
                \begin{pmatrix}
D(\infty) & 0_r\\
0_r & D(\infty)^{-\ast}
\end{pmatrix}
    \begin{pmatrix}
     -I_r & -\ii I_r\\
     -\ii I_r & I_r
    \end{pmatrix}            
    \begin{pmatrix}
D(\infty)^{-1} & 0_r\\
0_r & D(\infty)^{\ast}
\end{pmatrix},\\
    \widetilde{\Delta}_{1,-1}
    &=
            \begin{pmatrix}
D(\infty) & 0_r\\
0_r & D(\infty)^{-\ast}
\end{pmatrix}\\
&\times
\left[
    \frac{1}{48 h(-1)^2}
\begin{pmatrix}
     -3(h(-1)-h'(-1))I_r
     &
\ii\left(4h(-1)+3h'(-1)\right)I_r\\
\ii\left(4h(-1)+3h'(-1)\right)I_r
      &
      3(h(-1)-h'(-1))I_r
      \end{pmatrix}\right.\\
&\left.+
      \frac{1}{16h(-1)}
      \begin{pmatrix}
       L_{-1}^2 
       &
       -\ii L_{-1}(2\sqrt{2}I_r-L_{-1})\\
      \ii L_{-1}(2\sqrt{2}I_r+L_{-1}) 
       &
       -L_{-1}^2 
       \end{pmatrix}\right]\\
       &\times
        \begin{pmatrix}
        D(\infty)^{-1} & 0_r\\
        0_r & D(\infty)^\ast
        \end{pmatrix}
    \end{aligned}
\end{equation}

The first correction matrix $R_1(z)$ can be constructed as follows:
\begin{equation}\label{eq:R1residues}
    R_1(z)
    =
    \begin{cases}
        \displaystyle
        \frac{\Delta_{1,-2}}{(z-1)^2}+\frac{\Delta_{1,-1}}{z-1}
        +
        \frac{\widetilde{\Delta}_{1,-2}}{(z+1)^2}+\frac{\widetilde{\Delta}_{1,-1}}{z+1}, & 
        \,\, z\in \mathbb{C}\setminus\left(\overline{D_{\delta}(1)\cup D_{\delta}(-1)}\right),\\
        \displaystyle
        \frac{\Delta_{1,-2}}{(z-1)^2}+\frac{\Delta_{1,-1}}{z-1}
        +
        \frac{\widetilde{\Delta}_{1,-2}}{(z+1)^2}+\frac{\widetilde{\Delta}_{1,-1}}{z+1}-\Delta_1(z), & \,\,z\in D_{\delta}(1)\cup D_{\delta}(-1).
    \end{cases}
\end{equation}
Higher order terms can be constructed by iterating this procedure, for example
\begin{equation}
R_{2+}(z)
=
R_{2-}(z)
+
R_{1-}(z)\Delta_1(z)+\Delta_2(z), \qquad
z\in\partial D_{\delta}(1)\cup\partial D_{\delta}(-1),
\end{equation}
but in general this is a laborious task. We note that in this situation the order of the poles at $z=\pm 1$ increases faster than in other cases such as a Bessel parametrix, see \cite{KMcLVAV2004} for instance, and this complicates the calculation.

\section{Proof of Theorem \ref{th:MVOPs}}
In this section, we undo the transformations of the steepest descent method, in different regions of the complex plane, in order to obtain asymptotics for the original matrix $Y(z)$, and of the MVOPs as a direct consequence.
\subsection{Outer asymptotics}
Outside of the lens, we have $T(z)=S(z)=R(z)P^{(\infty)}(z)$, and then
\begin{equation}
\begin{aligned}
V(z)
&=
\begin{pmatrix}
\ee^{\frac{N\ell}{2}}I_r & 0_r\\
0_r & \ee^{-\frac{N\ell}{2}}I_r
\end{pmatrix}
R(z)P^{(\infty)}(z)
\begin{pmatrix} 
\ee^{N\left(g(z)-\frac{\ell}{2}\right)}I_r & 0_r\\
0_r & \ee^{-N\left(g(z)-\frac{\ell}{2}\right)}I_r
\end{pmatrix}.
\end{aligned}
\end{equation}
In particular,
\begin{equation}
\begin{aligned}
V_{11}(z)
&=
\begin{pmatrix}
\ee^{\frac{N\ell}{2}}I_r & 0_r\\
\end{pmatrix}
R(z)P^{(\infty)}(z)
\begin{pmatrix} 
\ee^{N\left(g(z)-\frac{\ell}{2}\right)}I_r \\ 0_r
\end{pmatrix}\\
&=
\ee^{Ng(z)}\left(R_{11}(z)P_{11}^{(\infty)}(z)+R_{12}(z)P_{21}^{(\infty)}(z)\right).
\end{aligned}
\end{equation}
We have
\begin{equation}
\begin{aligned}
P_{11}^{(\infty)}(z)
&=
\frac{1}{2}D(\infty)(\beta(z)+\beta(z)^{-1})D(z)^{-1},\\
P_{21}^{(\infty)}(z)
&=
\frac{\ii}{2}D(\infty)^{-\ast}(\beta(z)-\beta(z)^{-1})D(z)^{-1}.
\end{aligned}
\end{equation}
Since $R_{11}(z)=I_r+\mathcal{O}(N^{-1})$ and $R_{12}(z)=\mathcal{O}(N^{-1})$ as $N\to\infty$, and $P^{(\infty)}(z)$ is independent of $N$, we have
\begin{equation}
V_{11}(z)
=
\ee^{Ng(z)}\left(P_{11}^{(\infty)}(z)+\mathcal{O}(N^{-1})\right), \qquad N\to\infty.
\end{equation}
This gives strong asymptotics for the MVOPs away from the interval $[-1,1]$, since $V_{11}(z)=c^{-N}P_{N,N}(cz+d)$, using \eqref{eq:solRHPY} and \eqref{eq:VY}.

\subsection{Inner asymptotics}
If we take $z$ in the upper part of the lens, we have $T(z)=S(z)$, but we need to add a block triangular factor when we undo the transformations inside the lens:
\begin{equation}
\begin{aligned}
V(z)
&=
\begin{pmatrix}
\ee^{\frac{N\ell}{2}}I_r & 0_r\\
0_r & \ee^{-\frac{N\ell}{2}}I_r
\end{pmatrix}
S(z)
\begin{pmatrix} 
I_r & 0_r\\
\ee^{N\phi(z)}M(z)^{-1} & I_r
\end{pmatrix}
\begin{pmatrix} 
\ee^{N\left(g(z)-\frac{\ell}{2}\right)}I_r & 0_r\\
0_r & \ee^{-N\left(g(z)-\frac{\ell}{2}\right)}I_r
\end{pmatrix}\\
&=
\begin{pmatrix}
\ee^{\frac{N\ell}{2}}I_r & 0_r\\
0_r & \ee^{-\frac{N\ell}{2}} I_r
\end{pmatrix}
R(z)P^{(\infty)}(z)
\begin{pmatrix} 
\ee^{N\left(g(z)-\frac{\ell}{2}\right)}I_r & 0_r\\
M(z)^{-1}\ee^{N\left(\phi(z)+g(z)-\frac{\ell}{2}\right)} & \ee^{-N\left(g(z)-\frac{\ell}{2}\right)}I_r
\end{pmatrix}.
\end{aligned}
\end{equation}
In particular,
\begin{equation}
\begin{aligned}
V_{11}(z)
&=
\begin{pmatrix}
\ee^{\frac{N\ell}{2}}I_r & 0_r\\
\end{pmatrix}
R(z)P^{(\infty)}(z)
\begin{pmatrix} 
\ee^{N\left(g(z)-\frac{\ell}{2}\right)}I_r \\ \ee^{N\left(\phi(z)+g(z)-\frac{\ell}{2}\right)}M(z)^{-1}
\end{pmatrix}\\
&=
\ee^{N\left(g(z)+\frac{\phi(z)}{2}\right)}\\
&\times 
\left(\ee^{-\frac{N\phi(z)}{2}}\left(P_{11}^{(\infty)}(z)+\mathcal{O}(N^{-1})\right)
+\ee^{\frac{N\phi(z)}{2}}\left(P_{12}^{(\infty)}(z)M(z)^{-1}+\mathcal{O}(N^{-1})\right)
\right).
\end{aligned}
\end{equation}
 If we take the limit $z\to x$ from this side, we get
\begin{equation}
\begin{aligned}
V_{11+}(x)
&=
\ee^{N\left(g_{+}(x)+\frac{\phi_{+}(x)}{2}\right)}\\
&\times \left(\ee^{-\frac{N}{2}\phi_{+}(x)}\left(P_{11+}^{(\infty)}(x)+\mathcal{O}(N^{-1})\right)
+\ee^{\frac{N}{2}\phi_{+}(x)}\left(P_{12+}^{(\infty)}(x)M(x)^{-1}
+\mathcal{O}(N^{-1})\right)\right).
\end{aligned}
\end{equation}
Now, we note that $M(x)=D_{+}(x)D_+(x)^\ast=D_{-}(x)D_-(x)^\ast$, so 
\[
\begin{aligned}
P_{11+}^{(\infty)}(x)
&=
D(\infty)\widetilde{P}^{(\infty)}_{11+}(x)D_+(x)^{-1}, \\
P_{12+}^{(\infty)}(x)M(x)^{-1}
&=
D(\infty)\widetilde{P}^{(\infty)}_{12+}(x)D_-(x)^{\ast}D_-(x)^{-\ast}D_-(x)^{-1}
=
D(\infty)\widetilde{P}^{(\infty)}_{12+}(x)D_-(x)^{-1}.
\end{aligned}
\]
Also, since $M(x)$ is real valued on $[-1,1]$, we have $M(\overline{z})=\overline{M(z)}$ for $z\in\mathbb{C}\setminus[-1,1]$, and as a consequence $D_+(x)=\overline{D_-(x)}$. Therefore,
\begin{equation}\label{eq:U110}
\begin{aligned}
V_{11+}(x)
&=
\ee^{\frac{N}{2}(v(x)+\ell)} D(\infty)
\\
&\times \left(\ee^{-\frac{N}{2}\phi_{+}(x)}\left(\widetilde{P}^{(\infty)}_{11+}(x)D_+(x)^{-1}+\mathcal{O}(N^{-1})\right)
+\ee^{\frac{N}{2}\phi_{+}(x)}\left(\widetilde{P}^{(\infty)}_{12+}(x)\overline{D_+(x)^{-1}}
+\mathcal{O}(N^{-1})\right)\right)\\
&=
\ee^{\frac{N}{2}(v(x)+\ell)} D(\infty)
\\
&\times \left(\ee^{-\frac{N}{2}\phi_{+}(x)} \widetilde{P}^{(\infty)}_{11+}(x)D_+(x)^{-1}
+\ee^{\frac{N}{2}\phi_{+}(x)} \widetilde{P}^{(\infty)}_{12+}(x)\overline{D_+(x)^{-1}}
+\mathcal{O}(N^{-1})\right),
\end{aligned}
\end{equation}
since $|\ee^{\pm\frac{N}{2}\phi_{+}(x)}|=1$ for $x\in [-1,1]$, and $\widetilde{P}^{(\infty)}_+(x)$ and $D_+(x)$ are independent of $N$, and we have used that $\phi(z)=-2g(z)+v(z)+\ell$.

We define the function $\varphi(z)=z+(z^2-1)^{1/2}$ with a branch cut on $[-1,1]$, and on this interval we write
\begin{equation}
\varphi_{\pm}(x)
=
\ee^{\pm\ii \arccos x}
=
\ee^{\pm\ii\left(\frac{\pi}{2}-\arcsin x\right)},\qquad x\in[-1,1].
\end{equation}
As a consequence, we have
\begin{equation}
\begin{aligned}
\widetilde{P}^{(\infty)}_{11+}(x)
&=
\frac{\ee^{-\pi\ii/4}\varphi(x)_+^{1/2}}{\sqrt{2}(1-x^2)^{1/4}}
=
\frac{\ee^{-\frac{\ii}{2}\arcsin(x)}}{\sqrt{2}(1-x^2)^{1/4}}
\\
\widetilde{P}^{(\infty)}_{12+}(x)
&=
\frac{\ee^{\pi\ii/4}\varphi(x)_+^{-1/2}}{\sqrt{2}(1-x^2)^{1/4}}
=
\frac{\ee^{\frac{\ii}{2}\arcsin(x)}}{\sqrt{2}(1-x^2)^{1/4}},
\end{aligned}
\end{equation}
and therefore $\widetilde{P}^{(\infty)}_{11+}(x)=\overline{\widetilde{P}^{(\infty)}_{12+}(x)}$. So
\begin{multline}
\ee^{-\frac{N}{2}\phi_{+}(x)} \widetilde{P}^{(\infty)}_{11+}(x)D_+(x)^{-1}
+\ee^{\frac{N}{2}\phi_{+}(x)} \widetilde{P}^{(\infty)}_{12+}(x)\overline{D_+(x)^{-1}}\\
=
\frac{\sqrt{2}}{(1-x^2)^{1/4}}
\textrm{Re}\left(\ee^{-\frac{N}{2}\phi_{+}(x)-\frac{\ii}{2}\arcsin(x)} D_+(x)^{-1}\right).
\end{multline}

This leads to 
\begin{equation}\label{eq:leadingU11}
V_{11+}(x)
=
\frac{\sqrt{2}\ee^{\frac{N}{2}(v(x)+\ell)}}{(1-x^2)^{1/4}} D(\infty)
\left[\textrm{Re}
\left(
\ee^{-\ii\psi(x)}D_+(x)^{-1}\right)+\mathcal{O}(N^{-1})\right],
\end{equation}
where the phase function is given by \eqref{eq:psi}:
\begin{equation}
\psi(x)=-\frac{\ii N}{2}\phi_{+}(x)+\frac{1}{2}\arcsin(x).
\end{equation}
Again, we replace $V_{11+}(x)=c^{-N} P_{N,N}(cx+d)$ in order to deduce inner asymptotics for the MVOPs inside the interval $(-1,1)$, and away from the endpoints.
\subsection{Edge asymptotics}
If we consider $z$ inside $D_{\delta}(1)$ and inside the lens, we have
\begin{equation}
\begin{aligned}
V_{11}(z)
&=
\begin{pmatrix}
\ee^{\frac{N\ell}{2}}I_r & 0_r
\end{pmatrix}
R(z)P(z)
\begin{pmatrix} 
\ee^{N\left(g(z)-\frac{\ell}{2}\right)}I_r\\ 
\ee^{N\left(\phi(z)+g(z)-\frac{\ell}{2}\right)}M(z)^{-1}
\end{pmatrix}\\
&=
\ee^{Ng(z)}\\
&\times 
\left(R_{11}(z)P_{11}(z)+R_{12}(z)P_{21}(z)
+
\ee^{N\phi(z)}
\left(
R_{11}(z)P_{12}(z)+R_{12}(z)P_{22}(z)\right)M(z)^{-1}
\right)\\
&=
\ee^{N\left(g(z)+\frac{\phi(z)}{2}\right)}\\
&\times 
\left(\ee^{-\frac{N}{2}\phi(z)}\left(P_{11}(z)+\mathcal{O}(N^{-1})\right)+\ee^{\frac{N}{2}\phi(z)}\left(P_{12}(z)M(z)^{-1}
+\mathcal{O}(N^{-1})\right)\right).
\end{aligned}
\end{equation}
Using \eqref{eq:P1P} and \eqref{eq:P1}, we obtain
\begin{equation}
\begin{aligned}
\ee^{-\frac{N}{2}\phi(z)}P_{11}(z)
&=
\ee^{-\frac{N}{2}\phi(z)}
\begin{pmatrix}
I_r & 0_r
\end{pmatrix}
E_N(z)\Psi(z)
\begin{pmatrix}
\ee^{\frac{N}{2}\phi(z)}Q(z)^{-1}\\ 0_r 
\end{pmatrix}\\
&=
(E_{11}(z)\Psi_{11}(z)+E_{12}(z)\Psi_{21}(z))
Q(z)^{-1}\\
\ee^{\frac{N}{2}\phi(z)}P_{12}(z)M(z)^{-1}
&=
\ee^{\frac{N}{2}\phi(z)}
\begin{pmatrix}
I_r & 0_r
\end{pmatrix}
E_N(z)\Psi(z)
\begin{pmatrix}
0_r\\ \ee^{-\frac{N}{2}\phi(z)}Q(\overline{z})^{\ast}
\end{pmatrix}M(z)^{-1}
\\
&=
(E_{11}(z)\Psi_{12}(z)+E_{12}(z)\Psi_{22}(z))
Q(z)^{-1},
\end{aligned}
\end{equation}
$M(z)^{-1}=\left(Q(z)Q(\overline{z})^{\ast}\right)^{-1}=Q(\overline{z})^{-\ast}Q(z)^{-1}$. Also, we have
\begin{equation}
\begin{aligned}
E_{11}(z)
&=
\begin{pmatrix} I_r & 0_r\end{pmatrix}
P^{(\infty)}(z)
\begin{pmatrix}
Q(z) & 0_r\\
0_r & Q(\overline{z})^{-\ast}
\end{pmatrix}
\frac{1}{\sqrt{2}}
\begin{pmatrix}
I_r & -\ii I_r\\
-\ii I_r & I_r
\end{pmatrix}
(N^{2/3}f(z))^{\sigma_3/4}
\begin{pmatrix} I_r \\ 0_r \end{pmatrix}\\
&=
\frac{D(\infty)(N^{2/3}f(z))^{1/4}}{\sqrt{2}}
\begin{pmatrix} \widetilde{P}^{(\infty)}_{11}(z) & \widetilde{P}^{(\infty)}_{12}(z) \end{pmatrix}
\begin{pmatrix} D(z)^{-1}Q(z) \\ -\ii D(\overline{z})^\ast Q(\overline{z})^{-\ast} \end{pmatrix}
\end{aligned}
\end{equation}
and
\begin{equation}
\begin{aligned}
E_{12}(z)
&=
\begin{pmatrix} I_r & 0_r\end{pmatrix}
P^{(\infty)}(z)
\begin{pmatrix}
Q(z) & 0_r \\
0_r & Q(\overline{z})^{-\ast}
\end{pmatrix}
\frac{1}{\sqrt{2}}
\begin{pmatrix}
I_r & -\ii I_r\\
-\ii I_r & I_r
\end{pmatrix}
(N^{2/3}f(z))^{\sigma_3/4}
\begin{pmatrix} 0_r \\ I_r \end{pmatrix}\\
&=
\frac{D(\infty) (N^{2/3}f(z))^{-1/4}}{\sqrt{2}}
\begin{pmatrix} 
\widetilde{P}^{(\infty)}_{11}(z) & \widetilde{P}^{(\infty)}_{12}(z)\end{pmatrix}
\begin{pmatrix}-\ii D(z)^{-1}Q(z) \\  D(\overline{z})^\ast Q(\overline{z})^{-\ast} \end{pmatrix}\\
\end{aligned}
\end{equation}
Because of this calculation, the leading term as $N\to\infty$ corresponds to terms with $E_{11}(z)$. We write
\begin{equation}
\begin{aligned}
&
\left(E_{11}(z)\Psi_{11}(z)+E_{12}(z)\Psi_{21}(z)+\mathcal{O}(N^{-1})\right)Q(z)^{-1}\\
&=
\frac{N^{1/6}f(z)^{1/4}}{\sqrt{2}}D(\infty)
\left[\left(\widetilde{P}^{(\infty)}_{11}(z)D(z)^{-1}-\ii \widetilde{P}^{(\infty)}_{12} D(\overline{z})^\ast Q(\overline{z})^{-\ast}Q(z)^{-1}\right)\Psi_{11}(z)+\mathcal{O}(N^{-1/3})\right]\\
\end{aligned}
\end{equation}
and
\begin{equation}
\begin{aligned}
&\left(E_{11}(z)\Psi_{12}(z)+E_{12}(z)\Psi_{22}(z)+\mathcal{O}(N^{-1})\right)Q(z)^{-1}\\
&=
\frac{N^{1/6}f(z)^{1/4}}{\sqrt{2}}D(\infty)
\left[\left(\widetilde{P}^{(\infty)}_{11}(z)D(z)^{-1}-\ii \widetilde{P}^{(\infty)}_{12} D(\overline{z})^\ast Q(\overline{z})^{-\ast}Q(z)^{-1}\right)\Psi_{12}(z)+\mathcal{O}(N^{-1/3})\right].
\end{aligned}
\end{equation}
Now we take the limit $z\to x$ to the interval from this side (positive) and we get
\begin{equation}
\begin{aligned}
\widetilde{P}^{(\infty)}_{11+}(x)D_+(x)^{-1}-\ii \widetilde{P}^{(\infty)}_{12+} D_-(x)^\ast Q(x)^{-\ast}Q(x)^{-1}
&=
\widetilde{P}^{(\infty)}_{11+}(x)D_+(x)^{-1}-\ii \widetilde{P}^{(\infty)}_{12+} D_-(x)^{-1}\\
&=
\frac{\sqrt{2}\, \ee^{-\frac{\pi\ii}{4}}}{(1-x^2)^{1/4}}
\textrm{Re}\left(
\ee^{\frac{\pi\ii}{4}-\frac{\ii}{2}\arcsin(x)}
D_+(x)^{-1}\right),
\end{aligned}
\end{equation}
where we have used that  $M(x)=Q(x)Q(x)^{\ast}=D_-(x)D_-(x)^\ast$ on the interval, and the fact that $\widetilde{P}^{(\infty)}_{11+}(x)=\overline{\widetilde{P}^{(\infty)}_{12+}(x)}$ again. 

From \eqref{eq:localft}, we have $f(z)=C(z-1)\left(1+\mathcal{O}(z-1)\right)$ as $z\to 1$, with $C>0$, so $f(x)>0$ if $x>1$ and $f(x)<0$ if $x<1$. The fourth root of $f(z)$ is then taken with a cut on $(1-\delta,1)$, and positive in $(1,1+\delta)$. Then, on $(1-\delta,1)$, we have 
$(f(z)^{1/4})_+
=|f(x)|^{1/4}\ee^{\frac{\pi\ii}{4}}$, therefore $(f(z)^{1/4})_+\ee^{-\frac{\pi\ii}{4}}=|f(x)|^{1/4}$. 

Consequently,
\begin{equation}
\begin{aligned}
V_{11+}(x)
&=
\ee^{N\left(g_+(x)+\frac{\phi_+(x)}{2}\right)}
\frac{N^{1/6}|f(x)|^{1/4}}{(1-x^2)^{1/4}}D(\infty)\\
&\times \left(
\left(\Psi_{11}(x)+\Psi_{12}(x)\right)
\textrm{Re}\left(
\ee^{\frac{\pi\ii}{4}-\frac{\ii}{2}\arcsin(x)}
D_+(x)^{-1}\right)
+\mathcal{O}(N^{-1/3})\right).
\end{aligned}
\end{equation}

In this region, we have 
\begin{equation}
\begin{aligned}
\Psi_{11}(x)+\Psi_{12}(x)
=
-\sqrt{2\pi}\left(y_2(N^{2/3}f(x))+y_1(N^{2/3}f(x))\right)
&=
\sqrt{2\pi} y_0(N^{2/3}f(x)) \\
&=
\sqrt{2\pi}\textrm{Ai}(N^{2/3}f(x)).
\end{aligned}
\end{equation}

This gives the last part of Theorem \ref{th:MVOPs}. We can use this estimate in order to approximate the extreme zeros of MVOPs, in terms of the zeros of the Airy function. We recall that the zeros of $\textrm{Ai}(\zeta)$ are located on the negative real axis, which is mapped to the real axis with $x<1$.

\section{Proof of Theorem \ref{th:BnCn}}
We begin with the coefficient $C_{N,N}=Y_{N,12}^{(1)}Y_{N,21}^{(1)}$. In the sequel, we omit the subscript $N$ in the matrices for simplicity of notation. From 
\eqref{eq:U1Y1}, we have 
\begin{equation}
C_{N,N}
=
c^2\,
V_{12}^{(1)}V_{21}^{(1)}
\end{equation}
Then, undoing the transformations outside of the lens, we obtain
\begin{equation}
\begin{aligned}
V_{12}(z)
V_{21}(z)
&=
\ee^{-N(g(z)-\ell)}
\left(R(z)P^{(\infty)}(z)\right)_{12}
\ee^{N(g(z)-\ell)}
\left(R(z)P^{(\infty)}(z)\right)_{21}\\
&=
\left(R(z)P^{(\infty)}(z)\right)_{12}\left(R(z)P^{(\infty)}(z)\right)_{21},
\end{aligned}
\end{equation}
and since $V_{12}^{(1)}V_{21}^{(1)}=\lim_{z\to\infty}\left(z V_{12}(z)\right)\left(z V_{21}(z)\right)$, it follows that 
\begin{equation}
\begin{aligned}
C_{N,N}
&=
c^2\,\lim_{z\to\infty}\left(R_{11}(z) zP_{12}^{(\infty)}(z)+zR_{12}(z)P_{22}^{(\infty)}(z)\right)
\left(zR_{21}(z)P_{11}^{(\infty)}(z)+R_{22}(z)zP_{21}^{(\infty)}(z)\right).
\end{aligned}
\end{equation}

We use that $R(z)=I_{2r}+\mathcal{O}(z^{-1})$ as $z\to\infty$, and 
\begin{equation}\label{eq:asympPinfty}
P^{(\infty)}(z)
=
\begin{pmatrix}
D(\infty) & 0_r\\
0_r & D(\infty)^{-\ast}
\end{pmatrix}
\widetilde{P}^{(\infty)}(z)
\begin{pmatrix}
D(z)^{-1} & 0_r\\
0_r & D(\overline{z})^{\ast},
\end{pmatrix}
\end{equation}
with 
\begin{equation}\label{eq:asympPtinfty}
\widetilde{P}^{(\infty)}(z)
=
I_{2r}
+\frac{\ii}{2z} 
\begin{pmatrix}
0_r & I_r\\
-I_r & 0_r
\end{pmatrix}
+
\frac{1}{8z^2}I_{2r}+\mathcal{O}(z^{-3}), \qquad z\to\infty,
\end{equation}
to obtain
\begin{equation}
\begin{aligned}
\lim_{z\to\infty} R_{11}(z) zP_{12}^{(\infty)}(z)
=
\frac{\ii}{2}D(\infty)D(\infty)^\ast,\quad
\lim_{z\to\infty} R_{22}(z)z P_{21}^{(\infty)}(z)
=
-\frac{\ii}{2}D(\infty)^{-\ast}D(\infty)^{-1}.
\end{aligned}
\end{equation}

Therefore, the recurrence coefficient $C_{N,N}$ satisfies 
\begin{equation}
\begin{aligned}
C_{N,N}
&=
c^2\, \lim_{z\to\infty}
\left[\left(\frac{\ii}{2} D(\infty)D(\infty)^\ast+z R_{12}(z)\right)
\left(-\frac{\ii}{2} D(\infty)^{-\ast}D(\infty)^{-1}+z R_{21}(z)\right)\right]\\
&=
\frac{c^2}{4}\, \lim_{z\to\infty}
\left[\left(
I_r-2\ii z R_{12}(z)D(\infty)^{-\ast}D(\infty)^{-1}\right)
\left(I_r+2\ii z D(\infty)D(\infty)^{\ast}R_{21}(z)\right)\right]\\
\end{aligned}
\end{equation}
We have
\begin{equation}
\lim_{z\to\infty}
\left(z R_{12}(z)\right)=R_{12}^{(1)}, \qquad
\lim_{z\to\infty}
\left(z R_{21}(z)\right)=R_{21}^{(1)},
\end{equation}
and then
\begin{equation}
\begin{aligned}
C_{N,N}
&=
\frac{c^2}{4}
\left[I_r
+2\ii
\left(D(\infty)D(\infty)^{\ast}R_{21}^{(1)}
-R_{12}^{(1)}D(\infty)^{-\ast}D(\infty)^{-1}\right)\right]\\
\end{aligned}
\end{equation}

The matrix $R^{(1)}(z)$ admits an asymptotic expansion in inverse powers of $N$ and using \eqref{eq:R1residues}, we obtain
\begin{equation}
\begin{aligned}
R_{12}^{(1)}
&=
\frac{\left(\Delta_{1,-1}\right)_{12}+\left(\widetilde{\Delta}_{1,-1}\right)_{12}}{N}+\mathcal{O}(N^{-2})\\
R_{21}^{(1)}
&=
\frac{\left(\Delta_{1,-1}\right)_{21}+\left(\widetilde{\Delta}_{1,-1}\right)_{21}}{N}+\mathcal{O}(N^{-2}).
\end{aligned}
\end{equation}
Using \eqref{eq:Delta1}, we can easily check that
\begin{equation}
    -\left(\Delta_{1,-1}\right)_{12}
    D(\infty)^{-\ast}D(\infty)^{-1}
    +
    D(\infty)D(\infty)^{\ast}\left(\Delta_{1,-1}\right)_{21}
    =
    -\frac{\ii\sqrt{2}}{4h(1)}D(\infty)L_1D(\infty)^{-1},
\end{equation}
and similarly, using \eqref{eq:Deltat1}
\begin{equation}
    -\left(\widetilde{\Delta}_{1,-1}\right)_{12}
    D(\infty)^{-\ast}D(\infty)^{-1}
    +
    D(\infty)D(\infty)^{\ast}\left(\widetilde{\Delta}_{1,-1}\right)_{21}
    =
    \frac{\ii\sqrt{2}}{4h(-1)}D(\infty)L_{-1}D(\infty)^{-1},
\end{equation}

As a consequence, assembling everything together, we arrive at
\begin{equation}
\begin{aligned}
C_{N,N}
&=
\frac{c^2}{4}I_r+
\frac{\sqrt{2} c^2}{8N}
D(\infty)
\left(\frac{L_1}{h(1)}-\frac{L_{-1}}{h(-1)}\right)
D(\infty)^{-1}
+\mathcal{O}(N^{-2})\\
&=
\frac{c^2}{4}
\left[I_r+
\frac{\sqrt{2}}{2N}
D(\infty)
\left(\frac{L_1}{h(1)}-\frac{L_{-1}}{h(-1)}\right)
D(\infty)^{-1}
+\mathcal{O}(N^{-2})
\right],
\end{aligned}
\end{equation}
as $N\to\infty$, which gives the leading terms $C^{(0)}$ and $C^{(1)}$ in \eqref{eq:asympBnCn_theorem}.

Regarding the coefficient $B_{N,N}$, we use \eqref{eq:U1Y1} for $Y^{(1)}$ and the corresponding formula for the next term $Y^{(2)}$: 
\begin{equation}
\begin{aligned}
    Y^{(1)}
    &=
    c
\begin{pmatrix}
c^{N}I_r & 0_r\\
0_r & c^{-N} I_r
\end{pmatrix}
V^{(1)}
\begin{pmatrix}
c^{-N}I_r & 0_r\\
0_r & c^{N} I_r
\end{pmatrix}
-
N d\begin{pmatrix} I_r & 0_r\\ 0_r & -I_r \end{pmatrix},\\
Y^{(2)}
    &=
    c^2
\begin{pmatrix}
c^{N}I_r & 0_r\\
0_r & c^{-N} I_r
\end{pmatrix}
V^{(2)}
\begin{pmatrix}
c^{-N}I_r & 0_r\\
0_r & c^{N} I_r
\end{pmatrix}
-N d Y^{(1)}\sigma_3\\
&+
d Y^{(1)}-\frac{d^2}{2}
\begin{pmatrix}
N(N-1)I_r & 0_r\\
0_r & N(N+1) I_r
\end{pmatrix}.
\end{aligned}
\end{equation}
As a consequence,
\begin{equation}
    \begin{aligned}
    Y^{(1)}_{11}
    &=
    c V^{(1)}_{11}-Nd\\
    \left(Y^{(1)}_{12}\right)^{-1}
    &=
    \left(c^{2N+1}V^{(1)}_{12}\right)^{-1}
    =
    c^{-2N-1}
    \left(V^{(1)}_{12}\right)^{-1}\\
    \left(Y_{N,12}^{(2)}\right)^\ast
    &=
\left(c^{2N+2}V_{12}^{(2)}+Nd Y^{(1)}_{12}+d Y_{12}^{(1)}\right)^\ast\\
 &=
    c^{2N+2}\left(V_{12}^{(2)}\right)^{\ast}
    +(N+1)d\left(Y_{12}^{(1)}\right)^\ast\\
    &=
    c^{2N+2}\left(V_{12}^{(2)}\right)^{\ast}
    +(N+1)c^{2N+1}d\left(V_{12}^{(1)}\right)^\ast.
    \end{aligned}
\end{equation}
It follows that
\begin{equation}\label{eq:BnN0}
\begin{aligned}
    B_{N,N}
   &=
   c V^{(1)}_{11}-Nd
   -
   \left[
   c^{2N+2}\left(V_{12}^{(2)}\right)^{\ast}
    +(N+1)dc^{2N+1}\left(V_{12}^{(1)}\right)^\ast
   \right]
   c^{-2N-1}
    \left(V^{(1)}_{12}\right)^{-1}\\
    &=
    d+c\left(V^{(1)}_{11}-\left(V_{12}^{(2)}\right)^{\ast}
   \left(V^{(1)}_{12}\right)^{-1}\right)
    -(N+1)d\left(I_r+\left(V_{12}^{(1)}\right)^\ast
    \left(V^{(1)}_{12}\right)^{-1}\right).
    \end{aligned}
\end{equation}

In the transformations outside the lens, we have \begin{equation}
\begin{aligned}
    V(z)
    =
    \ee^{\frac{N\ell}{2}\sigma_3} T(z)
    \ee^{N(g(z)-\frac{\ell}{2})\sigma_3}
    &=
    \ee^{\frac{N\ell}{2}\sigma_3} S(z)\ee^{N(g(z)-\frac{\ell}{2})\sigma_3}\\
    &=
    \ee^{\frac{N\ell}{2}\sigma_3} R(z)P^{(\infty)}(z)
    \ee^{N(g(z)-\frac{\ell}{2})\sigma_3}.
    \end{aligned}
\end{equation}
Next, we write
\begin{equation}
\begin{aligned}
&V(z)
\begin{pmatrix}
    z^{-N} I_r & 0_r\\
    0_r & z^{N}I_r
\end{pmatrix}\\
&=
\ee^{\frac{N\ell}{2}\sigma_3} 
\left(I_{2r}+\frac{R^{(1)}}{z}+\frac{R^{(2)}}{z^2}+\mathcal{O}(z^{-3})\right)
\left(I_{2r}+\frac{P^{(1)}}{z}+\frac{P^{(2)}}{z^2}+\mathcal{O}(z^{-3})\right)
\ee^{-\frac{N\ell}{2}\sigma_3}\\
&\times 
\left(I_{2r}+\frac{G^{(1)}}{z}+\frac{G^{(2)}}{z^2}+\mathcal{O}(z^{-3})\right).
\end{aligned}
\end{equation}
Here
\begin{equation}
\ee^{Ng(z)\sigma_3}
\begin{pmatrix}
    z^{-N} I_r & 0_r\\
    0_r & z^{N}I_r
\end{pmatrix}
=
I_{2r}+\frac{G^{(1)}}{z}+\frac{G^{(2)}}{z^2}+\mathcal{O}(z^{-3}),
\end{equation}
where
\begin{equation}
    G_1
    =
    -N\mu_1 
    \begin{pmatrix}
        I_r & 0_r\\
        0_r & -I_r
    \end{pmatrix}, \qquad \mu_1=\int s\psi(s)\dd s,
\end{equation}
since 
\[
g(z)=\int \log(z-s)\psi(s)\dd s=\log(z)-\frac{1}{z}\int s\psi(s)\dd s-\frac{1}{2z^2} \int s^2\psi(s)\dd s+\mathcal{O}(z^{-3})
\]
as $z\to\infty$. The matrices $G^{(j)}$ for $j\geq 2$ are more complicated, but they all share the same block diagonal structure. Therefore, the factor $\ee^{-\frac{N\ell}{2}\sigma_3}$ commutes with all these factors, and we obtain
\begin{equation}
\begin{aligned}
&V(z)\begin{pmatrix}
    z^{-N} I_r & 0_r\\
    0_r & z^{N}I_r
\end{pmatrix}\\
&=
\ee^{\frac{N\ell}{2}\sigma_3} 
\left[I_{2r}+
\frac{1}{z}
\left(R^{(1)}+P^{(1)}+G^{(1)}\right)\right.\\
&\left.
+
\frac{1}{z^2}
\left(R^{(2)}+P^{(2)}+G^{(2)}
+
R^{(1)}P^{(1)}
+
R^{(1)}G^{(1)}
+
P^{(1)}G^{(1)}
\right)
+\mathcal{O}(z^{-3})\right]
\ee^{-\frac{N\ell}{2}\sigma_3}.
\end{aligned}
\end{equation}

As a consequence,
\begin{equation}\label{eq:U1U2}
\begin{aligned}
    V^{(1)}
    &=
    \ee^{\frac{N\ell}{2}\sigma_3} 
    \left(R^{(1)}+P^{(1)}+G^{(1)}\right)
    \ee^{-\frac{N\ell}{2}\sigma_3} \\
    V^{(2)}
    &=
    \ee^{\frac{N\ell}{2}\sigma_3} 
    \left(R^{(2)}+P^{(2)}+G^{(2)}
+
R^{(1)}P^{(1)}
+
R^{(1)}G^{(1)}
+
P^{(1)}G^{(1)}
\right)
    \ee^{-\frac{N\ell}{2}\sigma_3} \\
\end{aligned}
\end{equation}

In particular, 
\begin{equation}
\begin{aligned}
    V_{11}^{(1)}
    &=
    R_{11}^{(1)}+P_{11}^{(1)}+G_{11}^{(1)}
    =
    R_{11}^{(1)}+P_{11}^{(1)}-N\mu_1 I_r\\
    V_{12}^{(1)}
    &=
    \ee^{N\ell}\left(R_{12}^{(1)}+P_{12}^{(1)}+G_{12}^{(1)}\right)
    =
    \ee^{N\ell}\left(R_{12}^{(1)}+P_{12}^{(1)}\right).
\end{aligned}
\end{equation}

Also, if we write 
$D(z)=D(\infty)+z^{-1}D^{(1)}+\mathcal{O}(z^{-2})$ 
as $z\to\infty$, and we use \eqref{eq:Pinfty}, we obtain
\begin{equation}\label{eq:P11P12}
    \begin{aligned}
        P^{(1)}_{11}
        =   
        -D^{(1)} D(\infty)^{-1},\qquad
        P^{(1)}_{12}
        =
        \frac{\ii}{2}D(\infty)D(\infty)^\ast, \qquad
        P^{(1)}_{22}
        =   
        D(\infty)^{-\ast}\left(D^{(1)}\right)^\ast.
    \end{aligned}
\end{equation}

As a consequence, in the last term in \eqref{eq:BnN0}, we can write
\begin{equation}
\begin{aligned}
    &I_r+
    \left(V_{12}^{(1)}\right)^\ast
    \left(V^{(1)}_{12}\right)^{-1}\\
    &=
    I_r+\left(R_{12}^{(1)}+\frac{\ii}{2}D(\infty)D(\infty)^\ast\right)^\ast
    \left(R_{12}^{(1)}+\frac{\ii}{2}D(\infty)D(\infty)^\ast R_{12}^{(1)}\right)^{-1}\\
    &=
    I_r-\left(I_r-2\ii D(\infty)^{-\ast}D(\infty)^{-1}R_{12}^{(1)}\right)^\ast
    \left(I_r-2\ii R_{12}^{(1)}
    D(\infty)^{-\ast}D(\infty)^{-1}\right)^{-1}\\
    &=
    I_r-\left(I_r+2\ii \left(R_{12}^{(1)}\right)^\ast
    D(\infty)^{-\ast}D(\infty)^{-1}\right)
    \left(I_r-2\ii R_{12}^{(1)}
    D(\infty)^{-\ast}D(\infty)^{-1}\right)^{-1}.
\end{aligned}
\end{equation}
We have $R_{12}^{(1)}=N^{-1}\left((\Delta_{1,-1})_{12}+\widetilde{\Delta}_{1,-1})_{12}\right)+\mathcal{O}(N^{-2})$ as $N\to\infty$. Also,  using \eqref{eq:Delta1}, \eqref{eq:Deltat1}, it turns out that the matrices $L_1$ and $L_{-1}$ are Hermitian, and we can check that in fact $R_{12}^{(1)}+\left(R_{12}^{(1)}\right)^\ast=\mathcal{O}(N^{-2})$, and so $I_r+\left(U_{12}^{(1)}\right)^\ast
    \left(U^{(1)}_{12}\right)^{-1}=\mathcal{O}(N^{-2})$, and the last term in \eqref{eq:BnN0} is 
\begin{equation}
\begin{aligned}   
    -(N+1)d\left(I_r+\left(V_{12}^{(1)}\right)^\ast
    \left(V^{(1)}_{12}\right)^{-1}\right)
    =\mathcal{O}(N^{-1}).
\end{aligned}
\end{equation}

It is quite natural to conjecture that this term is in fact identically $0_r$, which would be consistent with the result in the scalar case. The difficulty is that this requires knowledge of higher order terms in the expansion for $R_N(z)$, in particular symmetry of the $(1,2)$ block with respect to conjugation, which is not immediately clear from the analysis so far.


Now we study the term
$V^{(1)}_{11}-\left(V_{12}^{(2)}\right)^{\ast}
   \left(V^{(1)}_{12}\right)^{-1}$.
From \eqref{eq:U1U2}, we obtain
   \begin{equation}
   \begin{aligned}
        V^{(2)}_{12}
    &=
    \ee^{N\ell} 
    \left(R^{(2)}_{12}+P^{(2)}_{12}+G^{(2)}_{12}
+
\left(R^{(1)}P^{(1)})\right)_{12}
+
\left(R^{(1)}G^{(1)}\right)_{12}
+
\left(P^{(1)}G^{(1)}\right)_{12}\right)
\end{aligned}
\end{equation}   

   First, we observe that $G^{(2)}_{12}=0_r$, since this matrix is block diagonal. Also,
\begin{equation}
    \begin{aligned}
        P^{(2)}_{12}
        &=
        \frac{\ii}{2}D(\infty)\left(D^{(1)}\right)^\ast\\
        \left(R^{(1)}P^{(1)}\right)_{12}
        &=
        R^{(1)}_{11}P^{(1)}_{12}+R^{(1)}_{12}P^{(1)}_{22}
        =
        \frac{\ii}{2}R^{(1)}_{11}D(\infty)D(\infty)^\ast
        +
        R^{(1)}_{12}D(\infty)^{-\ast}\left(D^{(1)}\right)^\ast,\\
        \left(R^{(1)}G^{(1)}\right)_{12}
        &=
        R^{(1)}_{11}G^{(1)}_{12}+R^{(1)}_{12}G^{(1)}_{22}
        =
        N\mu_1  R^{(1)}_{12},\\
        \left(P^{(1)}G^{(1)}\right)_{12}
        &=
        \frac{\ii N}{2}\mu_1 D(\infty)D(\infty)^\ast.
    \end{aligned}
\end{equation}
We then have
\begin{equation}
    \begin{aligned}
        &V^{(1)}_{11}-\left(V_{12}^{(2)}\right)^{\ast}
   \left(V^{(1)}_{12}\right)^{-1}\\
   &=
   R^{(1)}_{11}\\
   &+
   \left(N\mu_1+D^{(1)}D(\infty)^{-1}\right)\\
   &\times 
   \left[-I_r+\left(I_r+2\ii \left(R^{(1)}_{12}\right)^\ast D(\infty)^{-\ast}D(\infty)^{-1}\right)\left(I_r-2\ii R^{(1)}_{12} D(\infty)^{-\ast}D(\infty)^{-1}\right)^{-1}\right]\\
   &+
   2\ii 
   \left(R_{12}^{(2)}\right)^\ast
   D(\infty)^{-\ast}D(\infty)^{-1}
   \left(I_r-2\ii R^{(1)}_{12}D(\infty)^{-\ast}D(\infty)^{-1}\right)^{-1}.
   \end{aligned}
\end{equation}
Furthermore,
\begin{equation}
    \begin{aligned}
&\left[-I_r+\left(I_r+2\ii \left(R^{(1)}_{12}\right)^\ast D(\infty)^{-\ast}D(\infty)^{-1}\right)\left(I_r-2\ii R^{(1)}_{12} D(\infty)^{-\ast}D(\infty)^{-1}\right)^{-1}\right]\\
&=
2\ii\left[R_{12}^{(1)}+\left(R_{12}^{(1)}\right)^\ast\right]D(\infty)^{-\ast}D(\infty)^{-1}+\ldots,
    \end{aligned}
\end{equation}
where by $\ldots$ we indicate terms that include the product of at least two blocks of the matrix $R^{(1)}$. Since these are of order $\mathcal{O}(N^{-1})$, and since $R_{12}^{(1)}+\left(R_{12}^{(1)}\right)^\ast=\mathcal{O}(N^{-2})$, we obtain
\begin{equation}
    \begin{aligned}
        V^{(1)}_{11}-\left(V_{12}^{(2)}\right)^{\ast}
   \left(V^{(1)}_{12}\right)^{-1}
   &=
   \mathcal{O}(N^{-1})+
   \left(N\mu_1+D^{(1)}D(\infty)^{-1}\right)\mathcal{O}(N^{-2})+\mathcal{O}(N^{-1})\\
   &=
   \mathcal{O}(N^{-1}).
   \end{aligned}
\end{equation}

Writing everything together, we obtain $B_{N,N}=dI_r+\mathcal{O}(N^{-1})$, which confirms \eqref{eq:asympBnCn_theorem}.

\section{Proof of Theorem \ref{th:Hn}}

The norm of the MVOPs can be written directly in terms of the entries of the solution to the Riemann--Hilbert problem. Namely, 
\begin{equation}
\mathcal{H}_{N,N}
=
-2\pi\ii Y_{N,12}^{(1)}
=
-2\pi\ii \lim_{z\to\infty}\left(z^{N+1} Y_{N,12}(z)\right).
\end{equation}
Then, undoing the transformations outside of the lens, we obtain
\begin{equation}
\begin{aligned}
Y_{N,12}(z)
=
c^{2N+1} \ee^{-N(g(z)-\ell)}\left(R_{11}(z)P^{(\infty)}_{12}(z)+R_{12}(z)P^{(\infty)}_{22}(z)\right)
\end{aligned}
\end{equation}

Therefore, using that $\ee^{ng(z)}=z^{n}\left(1+\mathcal{O}(z^{-2})\right)$ as $z\to\infty$, it follows that
\begin{equation}
\begin{aligned}
\mathcal{H}_{N,N}
&=
-2\pi\ii\,
c^{2N+1} \ee^{N\ell}
\lim_{z\to\infty}\left[z^{N+1}\ee^{-Ng(z)}\left(R_{11}(z)P^{(\infty)}_{12}(z)+R_{12}(z)P^{(\infty)}_{22}(z)\right)\right]\\
&=
-2\pi\ii\,
c^{2N+1} \ee^{N\ell}
\lim_{z\to\infty}\left[R_{11}(z)z P^{(\infty)}_{12}(z)+zR_{12}(z)P^{(\infty)}_{22}(z)\right]\\
\end{aligned}
\end{equation}
We have
\begin{equation}
    \lim_{z\to\infty}R_{11}(z)z P^{(\infty)}_{12}
    =
    \frac{\ii}{2}D(\infty)D(\infty)^\ast, \qquad
    \lim_{z\to\infty}z R_{12}(z) P^{(\infty)}_{22}
    =
    R_{12}^{(1)}.
\end{equation}
We use the fact that
\begin{equation}
R_{12}^{(1)}
=
\frac{\left(\Delta_{1,-1}\right)_{12}+\left(\widetilde{\Delta}_{1,-1}\right)_{12}}{N}+\mathcal{O}(N^{-2})
\end{equation}
as $N\to\infty$, together with the explicit formulas for $\Delta_{1,-1}$ and $\widetilde{\Delta}_{1,-1}$ in \eqref{eq:Delta1} and \eqref{eq:Deltat1}, and we arrive at \eqref{eq:asympHNN}, with an $\mathcal{O}(N^{-1})$ correction given by \eqref{eq:H1}.

\section{Conclusions and outlook}
\label{sec:conclusions}
We have presented asymptotic results for matrix orthogonal polynomials with exponential weights on the real line. One essential ingredient, that was already used in \cite{DKR2023}, is the matrix Szeg\H{o} function, that we can construct explicitly in the present work. This allows us to avoid the complications of eigenvalue/eigenvector decomposition of the weight and subsequent asymptotic analysis. Nevertheless, the two approaches should be compatible, and understanding this connection in detail remains as a possible topic for further research.

The general assumptions that we have made, in particular the fact that the matrix part of the weight is independent of the large parameter $N$, allows us to follow several of the steps in the nonlinear steepest descent for scalar OPs with minor changes. Other modifications pertaining to the scalar part of the weight (two-cut case, or cases of singular equilibrium measure) can in principle be extended to the matrix case by making the corresponding  changes in the analysis.


One remaining task that looks more substantial to the theory of MVOPs is how to incorporate a dependence on $N$ in the matrix part. In this sense, many of the examples in the literature do not appear easily tractable. For example, if we start with 
\begin{equation}
    W(x)=\ee^{-x^2}Q(x)Q(x)^\ast,
\end{equation}
and we introduce the scaling $x\mapsto \sqrt{N} x$, then we obtain
\begin{equation}
    W(x)=\ee^{-N x^2}Q(\sqrt{N}x)Q(\sqrt{N}x)^\ast.
\end{equation}

Treating the $N$ dependent matrix factor $Q(\sqrt{N}x)$ in an asymptotic analysis seems complicated in general, even in the case $Q(\sqrt{N}x)=\ee^{\sqrt{N}A x}$. 

In this direction, we highlight \cite{BCD2012} as source of additional examples. Unlike in our case, there the limit behavior of recurrence coefficients is not a multiple of the identity matrix, and this points at additional complications in the asymptotic analysis. It is not clear at this stage how (or if it is possible) to set up a suitable equilibrium problem without resorting to eigenvalues and eigenvectors. 

Another topic that we have not addressed here in detail is the behavior of the zeros of MVOPs, that are understood as zeros of $\det P_{n,N}(x)$. In the $2\times 2$ case, with $v(x)=x^2$, $Q(x)=\ee^{Ax}$ and $A$ of the form \eqref{eq:A}, we can calculate the leading term in the asymptotics inside the interval \eqref{eq:asymp_inner}: namely, we have 
\begin{equation}\label{eq:Dplus22}
\begin{aligned}
D_+(x)^{-1}
&=
\frac{1}{2\sqrt{a^2+4}}
\begin{pmatrix}
4+a^2(1+\varphi_-(x)^2) & -2a\varphi_-(x)\\
-2a\varphi_-(x) & 4
\end{pmatrix},
\end{aligned}
\end{equation}
which, using that $\varphi_-(x)=x-\ii\sqrt{1-x^2}$ for $x\in(-1,1)$, leads to 
\begin{equation}\label{eq:Dplus22ReIm}
\begin{aligned}
\textrm{Re}\left(D_+(x)^{-1}\right)
&=
\frac{1}{\sqrt{a^2+4}}
\begin{pmatrix}
a^2x^2+2 & -ax\\
-ax & 2
\end{pmatrix},\\
\textrm{Im}\left(D_+(x)^{-1}\right)
&=
\frac{a\sqrt{1-x^2}}{\sqrt{a^2+4}}
\begin{pmatrix}
-ax & 1\\
1 & 0
\end{pmatrix}.
\end{aligned}
\end{equation}

Using \eqref{eq:Dplus22ReIm}, we can write the determinant and simplify:
\begin{equation}\label{eq:detI0}
\det\left[\textrm{Re}\left(\ee^{-\ii\psi(x)}D_+(x)^{-1}\right)\right]
=
\frac{4+a^2(2x^2-1)+(4+a^2)\cos(2\psi(x))}{2(4+a^2)},
\end{equation}
with $\psi(x)$ given by \eqref{eq:psi}. In Figures \ref{fig:inner2x2a} and \ref{fig:inner2x2b} we plot this function for different values of $N$. 

\begin{figure}[h!]
\centerline{
\includegraphics[scale=0.55]{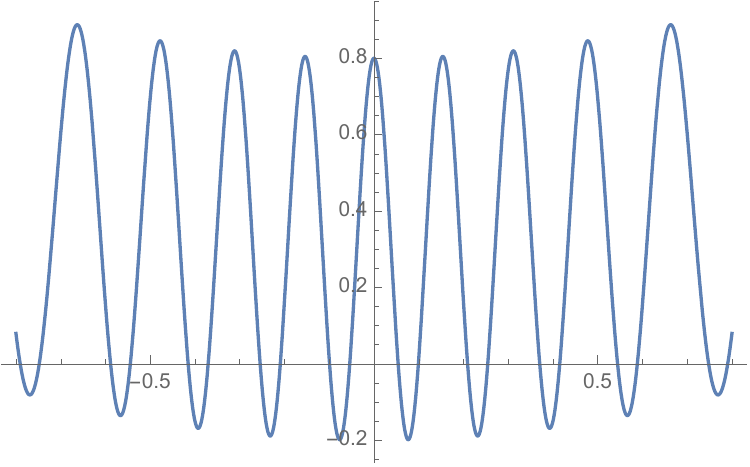}
\includegraphics[scale=0.55]{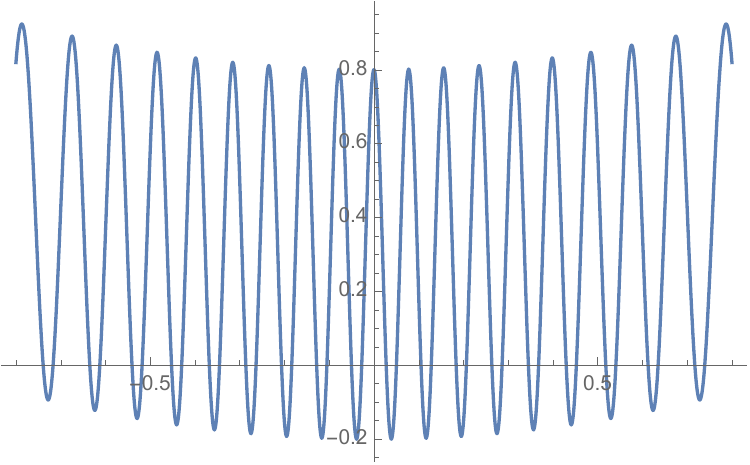}}
\caption{Leading term in the inner asymptotics in the $2\times 2 $ case, with $a=1$ and  $N=10$ (left) and $N=20$ (right).}
\label{fig:inner2x2a}
\end{figure}
\begin{figure}[h!]
\centerline{
\includegraphics[scale=0.55]{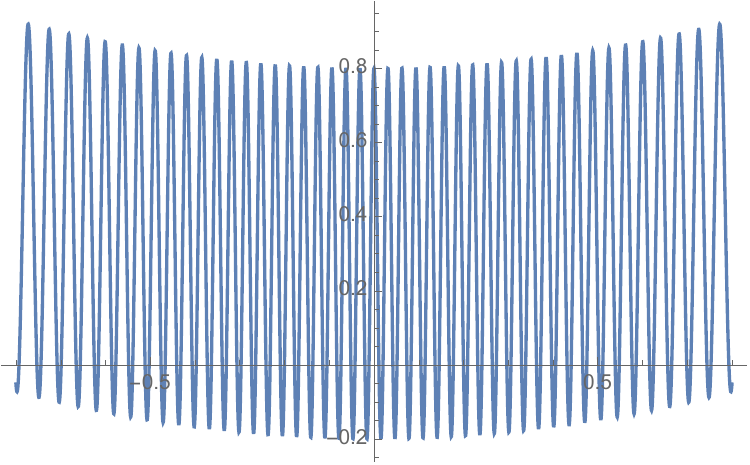}
\includegraphics[scale=0.55]{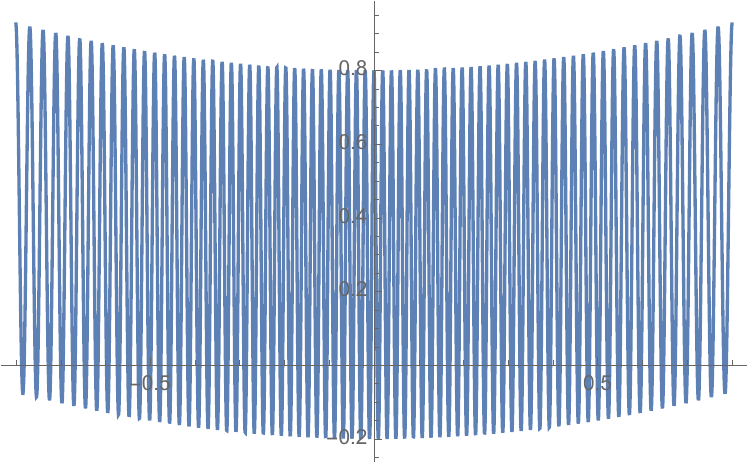}}
\caption{Leading term in the inner asymptotics in the $2\times 2 $ case, with $a=1$ and  $N=50$ (left) and $N=80$ (right).}
\label{fig:inner2x2b}
\end{figure}

In Figures \ref{fig:inner3x3a} and \ref{fig:inner3x3b} we plot the analogous leading term $\det\left[\textrm{Re}\left(\ee^{-\ii\psi(x)}D_+(x)^{-1}\right)\right]$, with $v(x)=x^2$, $Q(x)=\ee^{Ax}$ and $A$ of the form \eqref{eq:A}, in the $3\times 3$ case, again for different values of $N$ and $n=N$. It seems that the zeros of this determinant are naturally clustered in groups of $3$. A similar structure can be observed for higher values of $r$ (which is the size of the matrix part of the weight).

\begin{figure}[h!]
\centerline{
\includegraphics[scale=0.55]{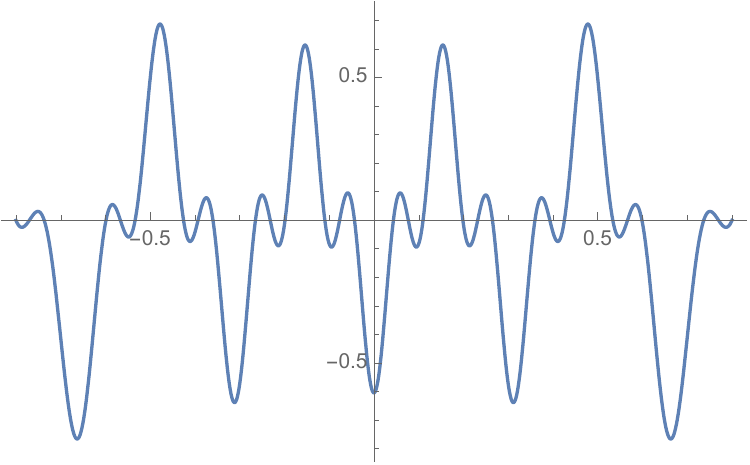}
\includegraphics[scale=0.55]{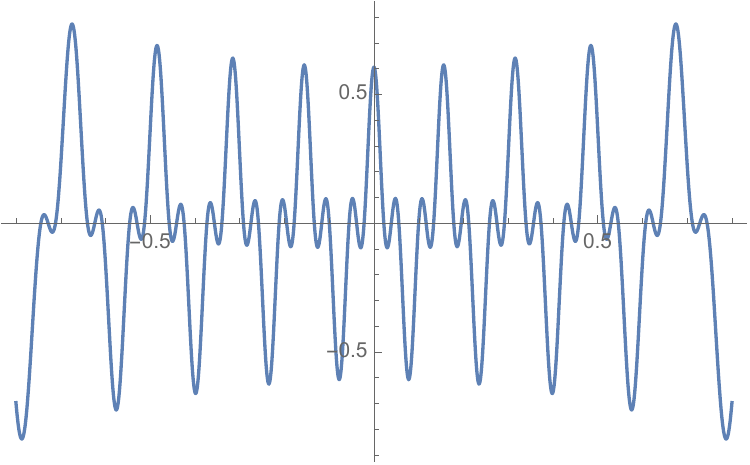}}
\caption{Leading term in the inner asymptotics in the $3\times 3 $ case, with $a=1$ and  $N=10$ (left) and $N=20$ (right).}
\label{fig:inner3x3a}
\end{figure}
\begin{figure}[h!]
\centerline{
\includegraphics[scale=0.55]{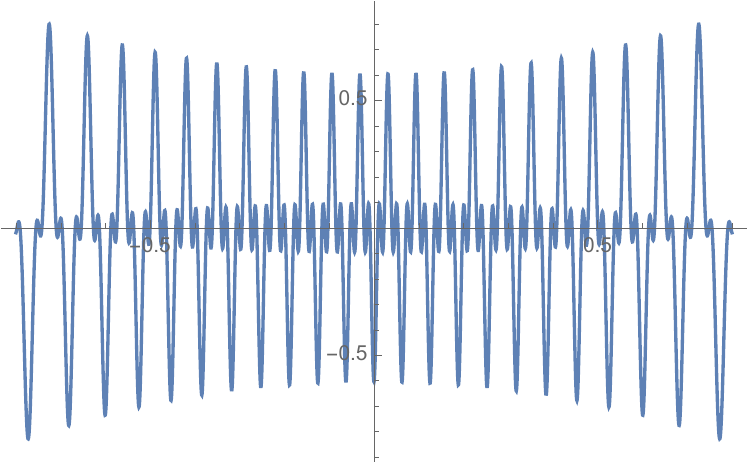}
\includegraphics[scale=0.55]{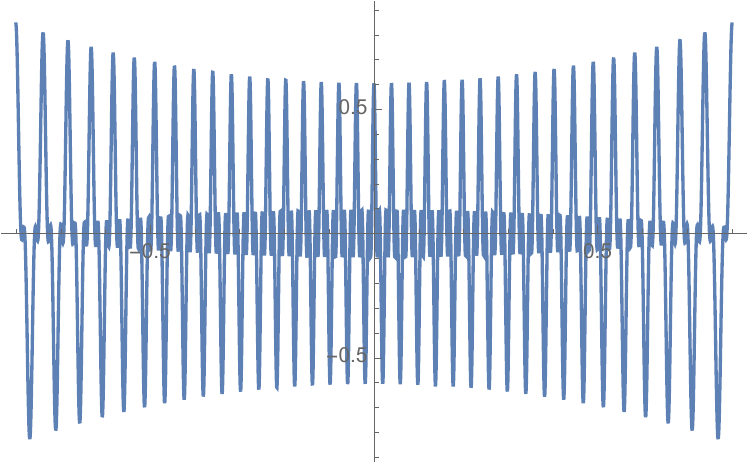}}
\caption{Leading term in the inner asymptotics in the $3\times 3 $ case, with $a=1$ and  $N=50$ (left) and $N=80$ (right).}
\label{fig:inner3x3b}
\end{figure}

\section*{Acknowledgements}
AD acknowledges financial support from Ministerio de Ciencia, Innovaci\'on y Universidades (Convocatoria de la Universidad Carlos III de Madrid de Ayudas para la recualificaci\'on del sistema universitario espa\~{n}ol para 2021-2023, de 1 de julio de 2021 en base al Real Decreto 289/2021, de 20 de abril de 2021). 
This work has been supported by Ministerio de Ciencia e Innovaci\'on (Spain), by Agencia Estatal de Investigación (AEI, Spain,
10.13039/501100011033), and by European Regional Development Fund (ERDF, A way of making Europe) through Grants No. PID2021-123969NB-I00 and
PID2021-122154NB-I00, as well as from Grant PID2024-159024NB-C21 funded by MCIU/AEI/10.13039/501100011033/FEDER, EU.

The authors would like to thank the two anonymous referees for their comments, that have resulted in an improved version of the paper.

\bibliography{MatrixHermite}

\begin{thebibliography}{10}

\bibitem{BI2005}
Pavel~M. Bleher and Alexander~R. Its.
\newblock Asymptotics of the partition function of a random matrix model.
\newblock {\em Ann. Inst. Fourier (Grenoble)}, 55(6):1943--2000, 2005.

\bibitem{BCD2012}
Jorge Borrego, Mirta Castro, and Antonio~J. Dur\'an.
\newblock Orthogonal matrix polynomials satisfying differential equations with
  recurrence coefficients having non-scalar limits.
\newblock {\em Integral Transforms Spec. Funct.}, 23(9):685--700, 2012.

\bibitem{CY2022}
W.~Riley Casper and M.~Yakimov.
\newblock The matrix {B}ochner problem.
\newblock {\em American Journal of Mathematics}, 144(4):1009--1065, 2022.

\bibitem{CCM2012}
G.A. Cassatella-Contra and M.~Mañas.
\newblock {R}iemann–{H}ilbert problems, matrix orthogonal polynomials and
  discrete matrix equations with singularity confinement.
\newblock {\em Studies in Applied Mathematics}, 128(3):252--274, 2012.

\bibitem{DER2021}
Alfredo Dea\~{n}o, Bruno Eijsvoogel, and Pablo Rom\'an.
\newblock Ladder relations for a class of matrix valued orthogonal polynomials.
\newblock {\em Studies in Applied Mathematics}, 146(2):463--497, 2021.

\bibitem{DKR2023}
Alfredo Dea\~{n}o, Arno B.~J. Kuijlaars, and Pablo Rom\'an.
\newblock Asymptotics of matrix valued orthogonal polynomials on $[-1,1]$.
\newblock {\em Advances in Mathematics}, 423:109043, 2023.

\bibitem{DKMVZ1999}
P.~Deift, T.~Kriecherbauer, K.~T-R McLaughlin, S.~Venakides, and X.~Zhou.
\newblock Strong asymptotics of orthogonal polynomials with respect to
  exponential weights.
\newblock {\em Comm. Pure Appl. Math.}, 52(12):1491--1552, 1999.

\bibitem{DD_2012}
Steven Delvaux and Holger Dette.
\newblock Zeros and ratio asymptotics for matrix orthogonal polynomials.
\newblock {\em J. Anal. Math.}, 118:657--690, 2012.

\bibitem{NIST:DLMF}
{\it NIST Digital Library of Mathematical Functions}.
\newblock \url{https://dlmf.nist.gov/}, Release 1.1.11 of 2023-09-15.
\newblock F.~W.~J. Olver, A.~B. {Olde Daalhuis}, D.~W. Lozier, B.~I. Schneider,
  R.~F. Boisvert, C.~W. Clark, B.~R. Miller, B.~V. Saunders, H.~S. Cohl, and
  M.~A. McClain, eds.

\bibitem{DG2004}
Antonio~J. Dur\'an and F.~Alberto Gr\"{u}nbaum.
\newblock Orthogonal matrix polynomials satisfying second-order differential
  equations.
\newblock {\em International Mathematics Research Notices}, 2004(10):461--484,
  01 2004.

\bibitem{DG2005}
Antonio~J. Dur\'an and F.~Alberto Gr\"unbaum.
\newblock Structural formulas for orthogonal matrix polynomials satisfying
  second-order differential equations. {I}.
\newblock {\em Constr. Approx.}, 22(2):255--271, 2005.

\bibitem{DG2007}
Antonio~J. Dur\'an and F.~Alberto Gr\"{u}nbaum.
\newblock Matrix orthogonal polynomials satisfying second-order differential
  equations: Coping without help from group representation theory.
\newblock {\em Journal of Approximation Theory}, 148(1):35--48, 2007.

\bibitem{ephremidze_2014}
Lasha Ephremidze.
\newblock An elementary proof of the polynomial matrix spectral factorization
  theorem.
\newblock {\em Proceedings of the Royal Society of Edinburgh Section A:
  Mathematics}, 144(4):747–751, 2014.

\bibitem{GdIM2012}
F.~Alberto Gr\"{u}nbaum, Manuel D.~de~la Iglesia, and Andrei
  Mart\'inez-Finkelshtein.
\newblock Properties of matrix orthogonal polynomials via their
  {R}iemann--{H}ilbert characterization.
\newblock {\em SIGMA}, 8(098):31 pages, 2012.

\bibitem{AB2020}
Mourad E.~H. Ismail and Walter Van~Assche, editors.
\newblock {\em Encyclopedia of special functions: the {A}skey-{B}ateman
  project. {V}ol. 1. {U}nivariate orthogonal polynomials}.
\newblock Cambridge University Press, Cambridge, 2020.

\bibitem{KMcL2000}
A.~B.~J. Kuijlaars and K.~T-R McLaughlin.
\newblock Generic behavior of the density of states in random matrix theory and
  equilibrium problems in the presence of real analytic external fields.
\newblock {\em Communications on Pure and Applied Mathematics}, 53(6):736--785,
  2000.

\bibitem{KMcLVAV2004}
A.~B.~J. Kuijlaars, K.~T.-R. McLaughlin, W.~Van~Assche, and M.~Vanlessen.
\newblock The {R}iemann-{H}ilbert approach to strong asymptotics for orthogonal
  polynomials on {$[-1,1]$}.
\newblock {\em Adv. Math.}, 188(2):337--398, 2004.

\bibitem{KT2009}
A.~B.~J. Kuijlaars and P.~M.~J. Tibboel.
\newblock The asymptotic behaviour of recurrence coefficients for orthogonal
  polynomials with varying exponential weights.
\newblock {\em J. Comput. Appl. Math.}, 233(3):775--785, 2009.

\bibitem{LL2001}
Eli Levin and Doron~S. Lubinsky.
\newblock {\em Orthogonal polynomials for exponential weights}, volume~4 of
  {\em CMS Books in Mathematics/Ouvrages de Math\'ematiques de la SMC}.
\newblock Springer-Verlag, New York, 2001.

\bibitem{ST1997}
Edward~B. Saff and Vilmos Totik.
\newblock {\em Logarithmic potentials with external fields}, volume 316 of {\em
  Grundlehren der mathematischen Wissenschaften [Fundamental Principles of
  Mathematical Sciences]}.
\newblock Springer-Verlag, Berlin, 1997.
\newblock Appendix B by Thomas Bloom.

\bibitem{Simon2a}
Barry Simon.
\newblock {\em Basic complex analysis}, volume Part 2A of {\em A Comprehensive
  Course in Analysis}.
\newblock American Mathematical Society, Providence, RI, 2015.

\bibitem{BS2015_4}
Barry Simon.
\newblock {\em Operator theory}.
\newblock A Comprehensive Course in Analysis. Part 4. American Mathematical
  Society, Providence, RI, 2015.

\bibitem{WM1957}
N.~Wiener and P.~Masani.
\newblock The prediction theory of multivariate stochastic processes: I. the
  regularity condition.
\newblock {\em Acta Math.}, 98:111--150, 1957.

\bibitem{Wimmer1986}
Harald~K. Wimmer.
\newblock Rellich's perturbation theorem on {H}ermitian matrices of holomorphic
  functions.
\newblock {\em J. Math. Anal. Appl.}, 114(1):52--54, 1986.

\bibitem{YK1978}
D.~Youla and N.~Kazanjian.
\newblock Bauer-type factorization of positive matrices and the theory of
  matrix polynomials orthogonal on the unit circle.
\newblock {\em IEEE Transactions on Circuits and Systems}, 25(2):57--69, 1978.

\end{thebibliography}
\bibliographystyle{plain}

\end{document}